\documentclass[11pt,a4paper]{article}

\usepackage[margin=1in]{geometry}
\usepackage[T1]{fontenc}
\usepackage{lmodern}
\IfFileExists{microtype.sty}{\usepackage{microtype}}{}
\usepackage{amsmath,amssymb,amsthm,mathtools}
\usepackage{booktabs}
\usepackage{float}
\usepackage{placeins}
\usepackage{enumitem}
\usepackage{array}
\usepackage{graphicx}
\usepackage{tikz}
\usetikzlibrary{arrows.meta,positioning}
\usepackage[colorlinks=true,linkcolor=blue,citecolor=blue,urlcolor=blue]{hyperref}
\usepackage{aliascnt}
\usepackage[nameinlink,noabbrev]{cleveref}

\newtheorem{theorem}{Theorem}[section]
\newaliascnt{proposition}{theorem}
\newtheorem{proposition}[proposition]{Proposition}
\aliascntresetthe{proposition}
\newaliascnt{lemma}{theorem}
\newtheorem{lemma}[lemma]{Lemma}
\aliascntresetthe{lemma}
\newaliascnt{corollary}{theorem}
\newtheorem{corollary}[corollary]{Corollary}
\aliascntresetthe{corollary}
\theoremstyle{definition}
\newaliascnt{definition}{theorem}
\newtheorem{definition}[definition]{Definition}
\aliascntresetthe{definition}
\newaliascnt{example}{theorem}
\newtheorem{example}[example]{Example}
\aliascntresetthe{example}
\newaliascnt{remark}{theorem}
\newtheorem{remark}[remark]{Remark}
\aliascntresetthe{remark}

\crefname{theorem}{Theorem}{Theorems}
\crefname{proposition}{Proposition}{Propositions}
\crefname{lemma}{Lemma}{Lemmas}
\crefname{corollary}{Corollary}{Corollaries}
\crefname{definition}{Definition}{Definitions}
\crefname{example}{Example}{Examples}
\crefname{remark}{Remark}{Remarks}

\newcommand{\Signs}{\{+1,-1\}}

\DeclareMathOperator{\rank}{rank}

\title{Signed GLMY Homology of Signed Graphs via Double Covers}
\author{Shuliang Bai$^{1}$\quad Jingyan Li$^{1,*}$\quad
Shing-Tung Yau$^{1,2}$\\[6pt]
\small $^{1}$Beijing Institute of Mathematical Sciences and Applications (BIMSA),
Beijing 101408, China\\
\small $^{2}$Yau Mathematical Sciences Center, Tsinghua University,
Beijing 100084, China\\
\small $^{*}$Corresponding author:
\href{mailto:jingyanli@bimsa.cn}{jingyanli@bimsa.cn}}
\date{}
\hypersetup{pdftitle={Signed GLMY Homology of Signed Graphs via Double Covers},
            pdfauthor={Shuliang Bai, Jingyan Li, and Shing-Tung Yau},
            pdfkeywords={signed graphs, GLMY homology, path homology,
              signed double covers, switching invariance, signed Laplacians}}

\begin{document}
\maketitle

\begin{abstract}
We define a signed GLMY chain complex over \(\mathbb R\) for signed
digraphs using sheet-labelled regular paths.  The complex is naturally
isomorphic to the deck anti-invariant subcomplex of the ordinary GLMY
complex on the signed double cover.  The double-cover realization yields
switching invariance and recovers ordinary GLMY homology for
switching-balanced signings.  Bidirected completion gives an
orientation-independent homology theory for signed graphs.  For a signed
graph, the zero-dimensional homology identifies with the kernel of the
signed Laplacian and has dimension equal to the number of balanced
connected components.
Signed GLMY homology is functorial under signed weak morphisms, which combine
vertex maps with switching functions and allow compatible arrow
contractions.  For signed digraphs, the all-positive reduction retains the
orientation sensitivity of ordinary GLMY homology, while explicit
computations show additional sensitivity to the arrow signs.  For a fixed
digraph with five vertices and nine arrows, we classify all \(512\) arrow
signings and obtain exactly four signed Betti vectors.  Precisely \(16\)
signings have nonzero second signed GLMY homology.
\end{abstract}



\section{Introduction}

A signed graph \(\Sigma=(G,\sigma)\) has edge labels in \(\Signs\).
Switching changes these labels while preserving their products around
cycles.  The resulting sign holonomy determines Harary balance and the
nullity of the signed Laplacian
\cite{Harary1953,Zaslavsky1982,Zaslavsky1991}.  Signed graphs also model
networks with cooperative and antagonistic interactions.

Grigor'yan, Lin, Muranov and Yau introduced the theory under the name
\emph{path homology} \cite{GLMY2012}.  It is now commonly referred to as
\emph{GLMY homology}, after the initials of the four authors; we use this
terminology throughout.
GLMY homology associates a chain complex to a digraph by selecting
linear combinations of allowed directed paths whose boundaries remain
allowed \cite{GLMY2012,GLMY2014Homotopy,GLMY2020}.  The GLMY boundary relations
depend on directions and shortcuts: a transitive triangle has an allowed
two-path whose boundary fills the triangle, whereas a directed three-cycle
has no nonzero boundary-preserving two-chain.  Different orientations of
the same graph can therefore have different GLMY homology.

We construct signed GLMY homology for digraphs and use bidirected
completion to define an orientation-independent invariant of signed
graphs.  For a signed digraph \((D,\sigma)\), the signed regular path
modules retain the sheet labels inherited by deleted faces.  A face can be
regular without being allowed, even when its base vertices are joined by
an arrow.  The signed GLMY complex consists of the allowed chains whose
boundaries remain allowed.  The signed complex is naturally isomorphic
to the deck anti-invariant subcomplex of the ordinary GLMY complex on the signed double
cover, giving
\[
        H_p^\sigma(D)\cong H_p(\widetilde D_\sigma)^-.
\]
This realization also reduces the computation of signed GLMY homology to
an ordinary GLMY computation on the double cover followed by restriction to
the deck anti-invariant summand.
Vertex switching induces an equivariant isomorphism of double covers.
Consequently, signed GLMY homology is switching invariant and agrees with
ordinary GLMY homology for switching-balanced signings.

For a signed graph \(\Sigma\), its bidirected completion
\(\overleftrightarrow\Sigma\) replaces each edge by the two opposite
arrows with the same sign.  Define
\[
        \mathcal H_p^\sigma(\Sigma)
        :=H_p^\sigma(\overleftrightarrow\Sigma)
        \cong H_p(\widetilde{\overleftrightarrow\Sigma}_\sigma)^-.
\]
The definition of \(\mathcal H_p^\sigma(\Sigma)\) requires no auxiliary
orientation.  The standard vertex inner product identifies
\(\mathcal H_0^\sigma(\Sigma)\) with the kernel of the signed Laplacian:
\[
        \mathcal H_0^\sigma(\Sigma)\cong\ker L^\sigma(\Sigma).
\]
Thus \(\dim\mathcal H_0^\sigma(\Sigma)\) counts balanced connected
components.  Higher groups also depend on the signing: two signings of the
bowtie graph have different first homology, and the signings of a fixed
five-vertex digraph yield four distinct Betti vectors in degrees zero,
one and two.  Signed weak morphisms induce homology maps compatible with
identities and composition.

Related homology theories for digraphs and quivers include singular and
cubical theories \cite{LiMuranovWuYau2026SingularQuivers}, cellular homology
\cite{TangYau2026Cellular}, the primitive variant of GLMY homology and circuit homology
\cite{LiMuranovWuYau2025Primitive,LiMuranovWuYau2025Circuits}, and weighted
variants of GLMY homology \cite{LinRenWangWu2019,MuranovSzczepkowskaVershinin2023}.
Weighted variants of GLMY homology modify coefficients through vertex
or path weights.  In the signed theory, arrow signs specify
one-dimensional real transports.  The transports define a rank-one
local system on the one-dimensional realization with one edge for each
arrow, so antiparallel arrows may carry independent signs.  The signed
boundary retains the relative sheet labels of deleted faces.  The signed
GLMY construction selects the anti-invariant part of the GLMY complex of
the signed double cover and differs from the equivariant construction
of Fu and Yau \cite{FuYau2026}.

We recall ordinary GLMY homology over a commutative ring \(K\) with
identity, writing \(H_p(D;K)\) when the coefficients need to be specified.
All signed chain spaces, signed homology groups, dimensions and Laplacians
in the remainder of the paper are over \(\mathbb R\).  We also omit the
coefficient symbol in ordinary real GLMY homology.  The arrow labels belong
to the multiplicative sign set \(\Signs\); they are distinct from the choice
of homology coefficients.  The identification of anti-invariant homology
with the anti-invariant part of cover homology uses the real projections
\(\frac12(\operatorname{id}\pm\tau)\).

Section~\ref{sec:signed-digraphs} develops the signed complex and its
cover realization.  Section~\ref{sec:canonical} treats undirected signed
graphs, Section~\ref{sec:functoriality} proves functoriality, and
Section~\ref{sec:h2-example} gives a complete sign-sensitive computation.
\section{Preliminaries}\label{sec:preliminaries}

\subsection{Signed graphs and switching}\label{sec:signed-graphs}

\begin{definition}
A signed graph is a pair $\Sigma=(G,\sigma)$, where $G=(V,E)$ is a finite simple graph and $\sigma:E\to\Signs$ is a sign function.
\end{definition}

A cycle in $G$ is a closed vertex sequence
\[
        C=v_0v_1\cdots v_{k-1}v_0, \qquad k\geq 3,
\]
such that $v_0,\ldots,v_{k-1}$ are pairwise distinct and
$v_rv_{r+1}\in E$ for $0\leq r\leq k-1$, where $v_k:=v_0$.
The sign of $C$ is
\[
        \sigma(C)=\prod_{r=0}^{k-1}\sigma(v_rv_{r+1}).
\]
The cycle $C$ is balanced if $\sigma(C)=+1$, and unbalanced if
$\sigma(C)=-1$.  The signed graph is balanced if every cycle is balanced.

\begin{definition}
For a function $\eta:V\to\Signs$, the switched signing is
\[
        \sigma^\eta(uv)=\eta(u)\sigma(uv)\eta(v).
\]
Two signings $\sigma$ and $\sigma'$ on the same underlying graph $G$
are switching equivalent if $\sigma'=\sigma^\eta$ for some
$\eta:V\to\Signs$.  The all-positive signing is the signing
$\sigma_+$ defined by $\sigma_+(e)=+1$ for every $e\in E$.
Switching preserves the sign of every cycle, and hence preserves balance.
\end{definition}

The following classical result is due to Harary \cite{Harary1953}.

\begin{theorem}[Harary's balance criterion]\label{thm:harary}
For a connected signed graph $\Sigma=(G,\sigma)$, the following are equivalent:
\begin{enumerate}[label=\textup{(\roman*)}]
    \item $\Sigma$ is balanced;
    \item $\Sigma$ is switching equivalent to the all-positive signing;
    \item $V(G)=V_+\sqcup V_-$, with either part allowed to be empty,
    such that positive edges lie within the parts and negative edges lie
    between the parts.
\end{enumerate}
\end{theorem}

\begin{definition}\label{def:signed-digraph}
A signed digraph is a pair $(D,\sigma)$, where \(D=(V,A)\) is a finite
loopless digraph with at most one arrow from \(i\) to \(j\) for each ordered
pair \((i,j)\), and \(\sigma:A\to\Signs\) is an arrow-sign function.
\end{definition}

\begin{definition}\label{def:switching-balanced-digraph}
Let \((D,\sigma)\) be a signed digraph.  We call \((D,\sigma)\)
\emph{switching-balanced} if there exists a function
\(\eta:V(D)\to\Signs\) such that
\[
        \sigma(ij)=\eta(i)\eta(j)
\]
for every arrow \(i\to j\).  Equivalently, after the switching
\[
        \sigma^\eta(ij)=\eta(i)\sigma(ij)\eta(j),
\]
all arrow signs are \(+1\).

If both arrows $i\to j$ and $j\to i$ are present, switching balance implies
\[
        \sigma(ij)=\eta(i)\eta(j)=\eta(j)\eta(i)=\sigma(ji).
\]
Antiparallel arrows with different signs therefore obstruct switching
balance, although they are permitted in a general signed digraph.
\end{definition}

\begin{remark}
\label{rem:directed-cycle-positivity}
Positivity of all directed cycles does not imply switching balance.
The transitive triangle with arrows $0\to1$, $1\to2$ and $0\to2$ has no
directed cycles.  If
$\sigma(01)\sigma(12)\sigma(02)=-1$, however, the equations
$\sigma(ij)=\eta(i)\eta(j)$ have no solution: multiplying the three
equations would give
$-1=\eta(0)^2\eta(1)^2\eta(2)^2=1$.
\end{remark}
\begin{remark}
\label{rem:underlying-signed-graph}
For a direction-compatible signed digraph, meaning that antiparallel arrows
have the same sign, forgetting orientations produces the underlying signed
graph $(U(D),\bar\sigma)$.  Here $U(D)$ has an undirected edge $\{i,j\}$
whenever $D$ contains at least one of the arrows $i\to j$ and $j\to i$, and
$\bar\sigma(\{i,j\})$ is the common sign of the arrows between $i$ and $j$
(or the sign of the unique such arrow when only one direction is present).
Applying Harary's criterion to each connected component therefore gives
\[
        (D,\sigma)\text{ is switching-balanced}
        \quad\Longleftrightarrow\quad
        (U(D),\bar\sigma)\text{ is balanced}.
\]
\end{remark}

\subsection{GLMY homology}\label{sec:glmy}

Let \(D=(V,A)\) be a finite loopless digraph, and write \(i\to j\) when
\((i,j)\in A\).  For \(p\ge0\), an elementary \(p\)-path is a sequence
\(i_0\cdots i_p\) of vertices.  The path is \emph{regular} when
consecutive vertices are distinct, and \emph{allowed} in \(D\) when
\[
        i_{r-1}\to i_r\qquad(1\le r\le p).
\]
Allowed paths are regular because \(D\) is loopless.  Nonconsecutive
vertices may repeat.

Let \(K\) be a commutative ring with identity.  Denote by
\(\mathcal R_p(D;K)\) the free \(K\)-module on all regular elementary
\(p\)-paths on \(V\), without imposing an arrow condition.  The allowed
path submodule of \(\mathcal R_p(D;K)\) is
\[
        \mathcal A_p(D;K)
        =\bigoplus_{i_0\cdots i_p\ \mathrm{allowed\ in}\ D}
          K e_{i_0\cdots i_p}.
\]
For \(p\ge1\), define the \(K\)-linear regular boundary
\[
        \partial_p:\mathcal R_p(D;K)\longrightarrow\mathcal R_{p-1}(D;K)
\]
on basis elements by
\[
        \partial_p e_{i_0\cdots i_p}
        =\sum_{q=0}^{p}(-1)^q
          e_{i_0\cdots\widehat{i_q}\cdots i_p},
\]
with every nonregular summand interpreted as zero.  In particular, deleting
an interior vertex \(i_q\) gives a zero term when \(i_{q-1}=i_{q+1}\).
Set \(\mathcal R_{-1}(D;K)=0\) and \(\partial_0=0\).  We write
\(\partial_*\) for the graded family \(\{\partial_p\}_{p\ge0}\).
The regular, unaugmented path boundaries satisfy
\[
        \partial_{p-1}\partial_p=0\qquad(p\ge1)
\]
\cite{GLMY2012,GLMY2020}.

The boundary of an allowed path need not be allowed: deleting an interior
vertex may produce consecutive vertices that are not joined by an arrow.
Define
\[
\begin{aligned}
        \Omega_0(D;K)&=\mathcal A_0(D;K),\\
        \Omega_p(D;K)&=\{u\in\mathcal A_p(D;K):
                    \partial_p u\in\mathcal A_{p-1}(D;K)\},\qquad p\ge1.
\end{aligned}
\]
For \(p\ge1\), write an allowed chain as
\(u=\sum_\gamma c_\gamma e_\gamma\), where the sum runs over allowed
elementary \(p\)-paths and \(c_\gamma\in K\).  The requirement
\(\partial_p u\in\mathcal A_{p-1}(D;K)\) means that, after collecting
equal basis paths in
\(\partial_p u=\sum_\gamma c_\gamma\partial_p e_\gamma\),
the total coefficient of each non-allowed regular \((p-1)\)-path is
zero.  The individual boundaries \(\partial_p e_\gamma\) need not be
allowed: contributions to the same non-allowed basis path can cancel
in the sum.

For \(u\in\Omega_1(D;K)\), the boundary \(\partial_1u\) belongs to
\(\Omega_0(D;K)=\mathcal A_0(D;K)\).  For \(p\ge2\) and
\(u\in\Omega_p(D;K)\), the chain \(\partial_pu\) is allowed and
\(\partial_{p-1}\partial_pu=0\in\mathcal A_{p-2}(D;K)\).  Hence
\(\partial_pu\in\Omega_{p-1}(D;K)\), and the restricted maps
\(\partial_p:\Omega_p(D;K)\to\Omega_{p-1}(D;K)\) form a chain complex.
We continue to denote the graded family of these restrictions by
\(\partial_*\).
With \(\Omega_{-1}(D;K)=0\), the GLMY homology groups are
\[
        H_p(D;K)
        :=\frac{\ker\bigl(\partial_p:\Omega_p(D;K)\to\Omega_{p-1}(D;K)\bigr)}
        {\operatorname{im}\bigl(\partial_{p+1}:\Omega_{p+1}(D;K)\to\Omega_p(D;K)\bigr)},
        \qquad p\ge0.
\]
For real coefficients we write \(\mathcal A_p(D)\), \(\Omega_p(D)\) and
\(H_p(D)\), and set \(\beta_p(D)=\dim_{\mathbb R}H_p(D)\).

For an undirected graph \(G=(V,E)\), replace every edge \(\{i,j\}\) by
both arrows \(i\to j\) and \(j\to i\), obtaining the bidirected digraph
\(\overleftrightarrow G\).  Define
\[
        H_p(G;K):=H_p(\overleftrightarrow G;K).
\]

\section{Signed digraph GLMY homology and sensitivity to orientation and signs}\label{sec:signed-digraphs}

Throughout this section, \((D,\sigma)\) denotes a signed digraph in the sense
of \cref{def:signed-digraph}.

\begin{definition}\label{def:signed-double-cover}
Let \((D,\sigma)\) be a signed digraph.  The signed double cover
\(\widetilde D_\sigma\) is the digraph with vertices
\(i^\varepsilon=(i,\varepsilon)\in V(D)\times\Signs\) and arrows
\[
        i^\varepsilon\to j^{\varepsilon\sigma(ij)}
        \qquad (i\to j\in A(D),\ \varepsilon\in\Signs).
\]
The deck involution is \(\tau(i^\varepsilon)=i^{-\varepsilon}\), so
\(\tau^2=\operatorname{id}\).  The covering projection is
\[
        \pi:\widetilde D_\sigma\longrightarrow D,
        \qquad \pi(i^\varepsilon)=i,
        \qquad \pi\circ\tau=\pi.
\]
\end{definition}

\begin{proposition}[Splitting criterion for the signed double cover]
\label{prop:double-cover-splitting}
Let $D_0$ be a weakly connected component of $D$.  The restriction of
$(D,\sigma)$ to $D_0$ is switching-balanced if and only if the signed double
cover of this restricted signed digraph is isomorphic over $D_0$ to
\[
        D_0^+\sqcup D_0^-,
\]
where $D_0^+$ and $D_0^-$ are two copies of $D_0$, each mapped
isomorphically onto $D_0$ by the covering projection.  If the restriction of $(D,\sigma)$ to $D_0$ is not switching-balanced,
then the signed double cover of this restriction is weakly connected.
\end{proposition}

\begin{proof}
Assume first that the restriction of $(D,\sigma)$ to $D_0$ is
switching-balanced.  Choose $\eta:V(D_0)\to\Signs$ such that
$\sigma(ij)=\eta(i)\eta(j)$ for every arrow $i\to j$ of $D_0$, and set
\[
    V_+:=\{i^{\eta(i)}:i\in V(D_0)\},
    \qquad
    V_-:=\{i^{-\eta(i)}:i\in V(D_0)\}.
\]
If $i\to j$ is an arrow of $D_0$, then
$\eta(i)\sigma(ij)=\eta(j)$.  Hence its two lifts are
\[
    i^{\eta(i)}\to j^{\eta(j)},
    \qquad
    i^{-\eta(i)}\to j^{-\eta(j)}.
\]
Thus no lifted arrow joins $V_+$ to $V_-$, and the induced subdigraphs on
$V_+$ and $V_-$ are both mapped isomorphically onto $D_0$ by the covering
projection.  The signed double cover therefore splits as
$D_0^+\sqcup D_0^-$.

Conversely, suppose that the signed double cover of the restriction to
$D_0$ splits over $D_0$ as two copies of $D_0$.  Choose one copy $C$.  For
each $i\in V(D_0)$ there is a unique $\eta(i)\in\Signs$ such that
$i^{\eta(i)}\in C$.  If $i\to j$ is an arrow of $D_0$, the lift contained in
$C$ has the form
\[
        i^{\eta(i)}\to j^{\eta(i)\sigma(ij)}.
\]
Its terminal vertex is also $j^{\eta(j)}$, so
$\eta(j)=\eta(i)\sigma(ij)$ and hence
$\sigma(ij)=\eta(i)\eta(j)$.  Thus the restriction to $D_0$ is
switching-balanced.

Finally, assume that the restriction to $D_0$ is not switching-balanced, and
let $C$ be a weakly connected component of the signed double cover of this
restriction.  Choose a
lift $v^\varepsilon\in C$ of some vertex $v\in V(D_0)$.  For any
$w\in V(D_0)$, choose a walk from $v$ to $w$ in the underlying undirected
graph of $D_0$.  Once the initial lift $v^\varepsilon$ is fixed, this walk
lifts uniquely in the underlying undirected graph of the double cover.
Hence $C$ contains at least one of the two lifts $w^+$ and $w^-$ of every
vertex $w$.

Suppose first that $C$ contains exactly one lift of each vertex of $D_0$.
Then $\tau C$ contains the other lift of each vertex, and $C$ and $\tau C$
are disjoint.  The covering projection restricts to an isomorphism from each
of these two components onto $D_0$, producing the splitting described in the
first part of the proposition.  This would imply that the restriction to
$D_0$ is switching-balanced, a contradiction.  Therefore $C$ contains both
lifts $v^+$ and $v^-$ of some vertex $v$.

Now let $w\in V(D_0)$ be arbitrary and choose an underlying walk from $v$
to $w$.  Lift this same walk once from $v^+$ and once from $v^-$.  The two
lifted walks are exchanged by the deck involution $\tau$; consequently,
if the first ends at $w^\delta$, the second ends at $w^{-\delta}$.  Both
lifted walks lie in $C$, so both $w^+$ and $w^-$ belong to $C$.  Since $w$
was arbitrary, $C$ contains both lifts of every vertex of $D_0$ and hence is
the entire signed double cover.  Therefore the signed double cover is weakly
connected.
\end{proof}

The deck involution changes every sheet label of an elementary path:
\[
        \tau\bigl(e_{v_0^{\varepsilon_0}\cdots v_p^{\varepsilon_p}}\bigr)
        =
        e_{v_0^{-\varepsilon_0}\cdots v_p^{-\varepsilon_p}},
\]
and extends linearly to regular path spaces.  The operator $\tau$
preserves arrows and commutes with deletion, so $\tau$ restricts to the
allowed path spaces and GLMY chain spaces and induces an involution on
homology.  We use $\tau$ for the induced operators as well.  For a real
vector space $C$ with involution $\tau$, the anti-invariant subspace is
\[
        C^-:=\ker(\tau+\operatorname{id}_C)
        =\{x\in C:\tau x=-x\}.
\]

\begin{definition}\label{def:signed-path-boundary}
A signed regular $p$-path over the lifted vertex set \(V(D)\times\Signs\) is a sequence
\[
        \widetilde\gamma=
        i_0^{\varepsilon_0}i_1^{\varepsilon_1}\cdots i_p^{\varepsilon_p},
        \qquad i_r\in V(D),\quad \varepsilon_r\in\Signs,
\]
such that the adjacent lifted vertices are distinct:
\[
        (i_{r-1},\varepsilon_{r-1})\neq (i_r,\varepsilon_r)
        \qquad (1\le r\le p).
\]
Thus \(u^+u^-\) and \(u^-u^+\) are regular, whereas \(u^+u^+\) and
\(u^-u^-\) are not.  Define
\(\mathcal R_p(\widetilde D_\sigma)\) to be the real vector space freely
generated by all such regular lifted \(p\)-paths; no arrow condition is imposed
in \(\mathcal R_p(\widetilde D_\sigma)\).  Define
\[
R_p^\sigma(D)
:=
\frac{\mathcal R_p(\widetilde D_\sigma)}
     {\displaystyle
      \bigl\langle
      e_{\tau\widetilde\gamma}+e_{\widetilde\gamma}:
      \widetilde\gamma\ \mathrm{regular}
      \bigr\rangle}.
\]
Write $[e_{\widetilde\gamma}]$ for the class of
$e_{\widetilde\gamma}$ in $R_p^\sigma(D)$.  Thus
$[e_{\tau\widetilde\gamma}]=-[e_{\widetilde\gamma}]$.
For $p\ge1$, define
\[
        \partial_{\mathrm{reg},p}^\sigma:
        R_p^\sigma(D)\longrightarrow R_{p-1}^\sigma(D)
\]
by
\[
        \partial_{\mathrm{reg},p}^\sigma
        [e_{i_0^{\varepsilon_0}\cdots i_p^{\varepsilon_p}}]
        =
        \sum_{q=0}^p
        (-1)^q
        [e_{i_0^{\varepsilon_0}\cdots
        \widehat{i_q^{\varepsilon_q}}\cdots
        i_p^{\varepsilon_p}}].
\]
Set \(R_{-1}^\sigma(D)=0\) and
\(\partial_{\mathrm{reg},0}^\sigma=0\).  We write
\(\partial_{\mathrm{reg},*}^\sigma\) for the graded family
\(\{\partial_{\mathrm{reg},p}^\sigma\}_{p\ge0}\).
A deletion term is zero when the remaining path has two adjacent
identical lifted vertices.  Equivalently, one first quotients the full
lifted path module by the nonregular paths and then imposes the deck
relation \([e_{\tau\widetilde\gamma}]=-[e_{\widetilde\gamma}]\).
The regular boundary on $\mathcal R_\bullet(\widetilde D_\sigma)$ commutes
with $\tau$ and therefore induces
\(\partial_{\mathrm{reg},*}^\sigma\) on $R_\bullet^\sigma(D)$, with
\[
        \partial_{\mathrm{reg},p-1}^\sigma
        \partial_{\mathrm{reg},p}^\sigma=0
        \qquad(p\ge1).
\]

A signed regular path $\widetilde\gamma$ is allowed if
\[
        i_{r-1}\to i_r\in A(D),
        \qquad
        \varepsilon_r=\varepsilon_{r-1}\sigma(i_{r-1}i_r)
        \quad (1\le r\le p).
\]
Let $\mathcal A_p^\sigma(D)\subseteq R_p^\sigma(D)$ be the linear subspace
generated by the classes \([e_{\widetilde\gamma}]\) of allowed signed regular
\(p\)-paths.
\end{definition}

For an allowed base $p$-path
\[
        \gamma=i_0i_1\cdots i_p,
\]
set
\[
        s_r=\prod_{t=0}^{r-1}\sigma(i_t i_{t+1})
        \qquad (1\le r\le p).
\]
By the definition of the signed double cover, for each
$\varepsilon\in\Signs$ the unique lift of $\gamma$ beginning at
$i_0^\varepsilon$ is
\[
        \widetilde\gamma^\varepsilon
        =i_0^\varepsilon i_1^{\varepsilon s_1}\cdots
        i_p^{\varepsilon s_p}.
\]
Thus the two lifts are
\[
        \widetilde\gamma^+
        =i_0^+i_1^{s_1}\cdots i_p^{s_p},
        \qquad
        \widetilde\gamma^-
        =i_0^-i_1^{-s_1}\cdots i_p^{-s_p}
        =\tau\widetilde\gamma^+,
\]
and every allowed lifted $p$-path arises uniquely in this way.

For an arbitrary regular lifted path $\widetilde\delta$, define
\[
\begin{aligned}
        \Lambda_p:R_p^\sigma(D)
        &\longrightarrow \mathcal R_p(\widetilde D_\sigma)^-,\\
        [e_{\widetilde\delta}]
        &\longmapsto
        e_{\widetilde\delta}-e_{\tau\widetilde\delta}.
\end{aligned}
\]
The deck relation
$[e_{\tau\widetilde\delta}]=-[e_{\widetilde\delta}]$ makes
$\Lambda_p$ a linear isomorphism.  Its inverse is
\[
        \Lambda_p^{-1}(z)=\tfrac12[z],
        \qquad z\in\mathcal R_p(\widetilde D_\sigma)^-,
\]
where $[z]$ denotes the image of $z$ in $R_p^\sigma(D)$.

\begin{proposition}[Anti-invariant realization]
\label{prop:anti-invariant-realization}
For every $p\ge0$,
\[
        \Lambda_p\bigl(\mathcal A_p^\sigma(D)\bigr)
        =\mathcal A_p(\widetilde D_\sigma)^-.
\]
Moreover, for $p\ge1$,
\[
        \Lambda_{p-1}\partial_{\mathrm{reg},p}^\sigma
        =\partial_p\Lambda_p
        \qquad\text{on }R_p^\sigma(D).
\]
\end{proposition}

\begin{proof}
Every allowed base path has exactly the two lifts $\widetilde\gamma^+$ and
$\widetilde\gamma^-$ defined immediately before the proposition, and the
deck involution exchanges them.  Conversely, every allowed lifted path is one
of these two lifts.  Hence $\Lambda_p$ identifies
$\mathcal A_p^\sigma(D)$ with
$\mathcal A_p(\widetilde D_\sigma)^-$.

For a regular lifted path $\widetilde\delta$, deletion of a vertex commutes
with $\tau$.  Therefore
\[
\begin{aligned}
\partial_p\Lambda_p[e_{\widetilde\delta}]
&=\partial_p\bigl(e_{\widetilde\delta}-e_{\tau\widetilde\delta}\bigr)\\
&=\Lambda_{p-1}\partial_{\mathrm{reg},p}^\sigma
  [e_{\widetilde\delta}],
\end{aligned}
\]
with the usual convention that a nonregular deleted face is zero.  This proves
the boundary identity.
\end{proof}

\begin{definition}\label{def:signed-glmy-complex}
Define
\[
        \Omega_0^\sigma(D)=\mathcal A_0^\sigma(D),
        \qquad
        \Omega_p^\sigma(D)
        =
        \{z\in \mathcal A_p^\sigma(D):
        \partial_{\mathrm{reg},p}^\sigma z
        \in \mathcal A_{p-1}^\sigma(D)\},
        \quad p\ge1.
\]
For $p\ge1$, define
\[
        \partial_p^\sigma:
        \Omega_p^\sigma(D)\longrightarrow\Omega_{p-1}^\sigma(D),
        \qquad
        \partial_p^\sigma z
        :=\partial_{\mathrm{reg},p}^\sigma z,
\]
and set \(\partial_0^\sigma=0\).  We write
\(\partial_*^\sigma\) for the graded signed GLMY boundary when no
single degree is being emphasized.  Thus $\Omega_p^\sigma(D)$ consists
exactly of the allowed signed $p$-chains whose alternating-deletion
boundary is again allowed.
\end{definition}

\begin{remark}\label{rem:sheet-labelled-boundary}
The signed formal boundary distinguishes faces with the same base vertices
but different relative sheet labels.  In particular, a deleted face need
not be a scalar multiple of the allowed lift of the corresponding shortcut.
Suppose that $i\to j\to k$ and $i\to k$ are arrows but
\[
        \sigma(ik)\neq\sigma(ij)\sigma(jk).
\]
Deleting $j$ from the lifted two-step path gives a face with relative
sheet change $\sigma(ij)\sigma(jk)$, whereas the allowed lifted shortcut
has relative sheet change $\sigma(ik)$.  The inherited face and the
allowed shortcut therefore belong to different deck orbits and represent
linearly independent classes in $R_1^\sigma(D)$.  More generally, for
$p\ge1$, choose one representative from each deck orbit of regular
$(p-1)$-paths.  An
allowed signed chain $z$ belongs to $\Omega_p^\sigma(D)$ precisely when
the coefficient of every non-allowed basis class in
$\partial_{\mathrm{reg},p}^\sigma z$ is zero.

Consider the transitive triangle
\[
        0\to1,\qquad 1\to2,\qquad 0\to2,
\]
with
\[
        \sigma(01)=+1,\qquad \sigma(12)=+1,\qquad \sigma(02)=-1.
\]
The signed triangle and its double cover are shown in
\cref{fig:signed-triangle-lift-obstruction}.  Deleting the middle vertex
of the allowed path $0^+\to1^+\to2^+$ gives $[e_{0^+2^+}]$.
The class $[e_{0^+2^+}]$ does not belong to $\mathcal A_1^\sigma(D)$:
the lift of $0\to2$ starting at $0^+$ ends at $2^-$ because
$\sigma(02)=-1$.

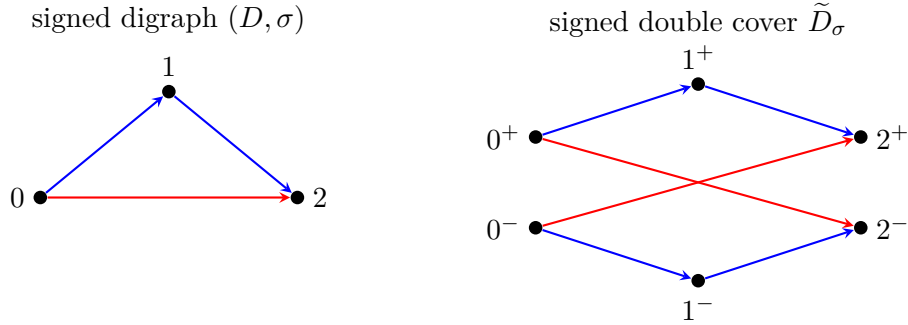
\begin{figure}[H]
\centering
\begin{tikzpicture}[>=stealth,scale=1.0,
    vtx/.style={circle,fill=black,inner sep=1.8pt},
    posedge/.style={->,blue,thick},
    negedge/.style={->,red,thick}]
\begin{scope}[xshift=0cm]
  \node at (0,2.10) {signed digraph $(D,\sigma)$};
  \node[vtx] (b0) at (-1.7,-.25) {};
  \node[vtx] (b1) at (0,1.15) {};
  \node[vtx] (b2) at (1.7,-.25) {};
  \node[left=2pt]  at (b0) {$0$};
  \node[above=2pt] at (b1) {$1$};
  \node[right=2pt] at (b2) {$2$};
  \draw[posedge] (b0)--(b1);
  \draw[posedge] (b1)--(b2);
  \draw[negedge] (b0)--(b2);
\end{scope}
\begin{scope}[xshift=7.0cm]
  \node at (0,2.10) {signed double cover $\widetilde D_\sigma$};
  \node[vtx] (p0) at (-2.15,.55) {};
  \node[vtx] (p1) at (0,1.25) {};
  \node[vtx] (p2) at (2.15,.55) {};
  \node[vtx] (m0) at (-2.15,-.65) {};
  \node[vtx] (m1) at (0,-1.35) {};
  \node[vtx] (m2) at (2.15,-.65) {};
  \node[left=2pt]  at (p0) {$0^+$};
  \node[above=2pt] at (p1) {$1^+$};
  \node[right=2pt] at (p2) {$2^+$};
  \node[left=2pt]  at (m0) {$0^-$};
  \node[below=2pt] at (m1) {$1^-$};
  \node[right=2pt] at (m2) {$2^-$};
  \draw[posedge] (p0)--(p1);
  \draw[posedge] (p1)--(p2);
  \draw[posedge] (m0)--(m1);
  \draw[posedge] (m1)--(m2);
  \draw[negedge] (p0)--(m2);
  \draw[negedge] (m0)--(p2);
\end{scope}
\end{tikzpicture}
\caption{The signed transitive triangle and its signed double cover.  Blue
arrows project to positive arrows of $D$, while red arrows project to the
negative shortcut $0\to2$.}
\label{fig:signed-triangle-lift-obstruction}
\end{figure}
Then
\[
\begin{aligned}
        \mathcal A_0^\sigma(D)
        &=\bigl\langle [e_{0^+}],[e_{1^+}],[e_{2^+}]\bigr\rangle,\\
        \mathcal A_1^\sigma(D)
        &=\bigl\langle [e_{0^+1^+}],[e_{1^+2^+}],[e_{0^+2^-}]\bigr\rangle,\\
        \mathcal A_2^\sigma(D)
        &=\bigl\langle [e_{0^+1^+2^+}]\bigr\rangle.
\end{aligned}
\]
The displayed bases use lifts starting in the positive sheet.  Their
deck-conjugate lifts represent the negatives of these classes; for example,
$[e_{0^-2^+}]=-[e_{0^+2^-}]$.
For the signed triangle, $\Omega_0^\sigma(D)=\mathcal A_0^\sigma(D)$
by definition, and the boundary of every allowed signed arrow belongs
to $\mathcal A_0^\sigma(D)$.  Hence
\[
        \Omega_0^\sigma(D)=\mathcal A_0^\sigma(D),
        \qquad
        \Omega_1^\sigma(D)=\mathcal A_1^\sigma(D).
\]
Indeed,
\[
\begin{aligned}
        \partial_1^\sigma[e_{0^+1^+}]
        &=[e_{1^+}]-[e_{0^+}],\\
        \partial_1^\sigma[e_{1^+2^+}]
        &=[e_{2^+}]-[e_{1^+}],\\
        \partial_1^\sigma[e_{0^+2^-}]
        &=[e_{2^-}]-[e_{0^+}]
          =-[e_{2^+}]-[e_{0^+}].
\end{aligned}
\]
However, the boundary of the sole degree-two generator is
\[
        \partial_{\mathrm{reg},2}^\sigma[e_{0^+1^+2^+}]
        =
        [e_{1^+2^+}]-[e_{0^+2^+}]+[e_{0^+1^+}].
\]
The first and third classes belong to $\mathcal A_1^\sigma(D)$, but the middle
class $[e_{0^+2^+}]$ does not: the arrow $0\to2$ is negative, so its allowed
lift represents $[e_{0^+2^-}]$.  Hence
\[
        \partial_{\mathrm{reg},2}^\sigma[e_{0^+1^+2^+}]
        \notin \mathcal A_1^\sigma(D),
        \qquad
        [e_{0^+1^+2^+}]\notin\Omega_2^\sigma(D).
\]
Since \(\mathcal A_2^\sigma(D)\) is one-dimensional and its generator
has a nonzero boundary coefficient at the non-allowed class
$[e_{0^+2^+}]$, the boundary-preserving subspace is
\[
        \Omega_2^\sigma(D)=0.
\]
\end{remark}

\begin{proposition}\label{prop:signed-glmy-complex}
For every signed digraph $(D,\sigma)$, the maps
\[
        \partial_p^\sigma:\Omega_p^\sigma(D)\to\Omega_{p-1}^\sigma(D),
        \qquad p\ge1,
\]
satisfy
\[
        \partial_{p-1}^\sigma\partial_p^\sigma=0
        \qquad(p\ge1).
\]
Thus $(\Omega_\bullet^\sigma(D),\partial_*^\sigma)$ is a chain complex.
\end{proposition}

\begin{proof}
If $z\in\Omega_1^\sigma(D)$, then
$\partial_1^\sigma z\in\mathcal A_0^\sigma(D)=\Omega_0^\sigma(D)$.
For $p\ge2$ and $z\in\Omega_p^\sigma(D)$, the chain
$\partial_p^\sigma z=\partial_{\mathrm{reg},p}^\sigma z$ belongs to
$\mathcal A_{p-1}^\sigma(D)$.  Because $\partial_p^\sigma z\in\mathcal A_{p-1}^\sigma(D)$ and
\[
\partial_{\mathrm{reg},p-1}^\sigma
\bigl(\partial_p^\sigma z\bigr)
=\partial_{\mathrm{reg},p-1}^\sigma
\partial_{\mathrm{reg},p}^\sigma z=0,
\]
the chain $\partial_p^\sigma z$ belongs to $\Omega_{p-1}^\sigma(D)$.
On $\Omega_{p-1}^\sigma(D)$ the restricted boundary agrees with the ambient
regular boundary, and therefore
$\partial_{p-1}^\sigma\partial_p^\sigma z=0$.
\end{proof}

\begin{proposition}\label{prop:degree-one-incidence}
Let $i\to j$ be an arrow of a signed digraph $(D,\sigma)$.  The allowed
lift beginning at $i^+$ determines the class
$[e_{i^+j^{\sigma(ij)}}]\in\mathcal A_1^\sigma(D)$, and
\[
        \partial_1^\sigma[e_{i^+j^{\sigma(ij)}}]
        =
        \sigma(ij)[e_{j^+}]-[e_{i^+}].
\]
Thus, with respect to the vertex basis $\{[e_{v^+}]:v\in V(D)\}$, the
degree-one signed GLMY boundary is the twisted incidence boundary.
\end{proposition}

\begin{proof}
The allowed lift of the arrow $i\to j$ that begins at $i^+$ is
$i^+\to j^{\sigma(ij)}$.  Applying the signed boundary gives
\[
        \partial_1^\sigma[e_{i^+j^{\sigma(ij)}}]
        =
        [e_{j^{\sigma(ij)}}]-[e_{i^+}].
\]
By the deck relation in $R_0^\sigma(D)$, every zero-path class satisfies
\[
        [e_{v^\varepsilon}]=\varepsilon[e_{v^+}].
\]
Substituting $\varepsilon=\sigma(ij)$ in the first term therefore yields
\[
        \partial_1^\sigma[e_{i^+j^{\sigma(ij)}}]
        =
        \sigma(ij)[e_{j^+}]-[e_{i^+}],
\]
as claimed.
\end{proof}

\begin{definition}\label{def:signed-digraph-homology}
The signed GLMY homology of $(D,\sigma)$ is the homology of the chain complex
\[
        \cdots
        \xrightarrow{\partial_3^\sigma}
        \Omega_2^\sigma(D)
        \xrightarrow{\partial_2^\sigma}
        \Omega_1^\sigma(D)
        \xrightarrow{\partial_1^\sigma}
        \Omega_0^\sigma(D).
\]
With \(\Omega_{-1}^\sigma(D)=0\), the signed GLMY homology groups are
\[
 H_p^\sigma(D)
 =\frac{\ker\bigl(\partial_p^\sigma:\Omega_p^\sigma(D)
                               \to\Omega_{p-1}^\sigma(D)\bigr)}
        {\operatorname{im}\bigl(\partial_{p+1}^\sigma:\Omega_{p+1}^\sigma(D)
                               \to\Omega_p^\sigma(D)\bigr)},\qquad p\ge0.
\]
We set \(\beta_p^\sigma(D)=\dim_{\mathbb R}H_p^\sigma(D)\).
\end{definition}

\begin{proposition}[Double-cover description]\label{prop:double-cover-description}
For every signed digraph $(D,\sigma)$ and every \(p\ge0\),
\[
 H_p^\sigma(D)\cong H_p(\widetilde D_\sigma)^-
 =\{\alpha\in H_p(\widetilde D_\sigma):\tau\alpha=-\alpha\}.
\]
The isomorphism is induced by the chain isomorphism
\(\Lambda_\bullet:\Omega_\bullet^\sigma(D)
\longrightarrow\Omega_\bullet(\widetilde D_\sigma)^-\).
\end{proposition}

\begin{proof}
By \cref{prop:anti-invariant-realization}, $\Lambda_p$ identifies
$\mathcal A_p^\sigma(D)$ with
$\mathcal A_p(\widetilde D_\sigma)^-$ and commutes with the formal boundaries.
It follows that, for $z\in\mathcal A_p^\sigma(D)$,
\[
 z\in\Omega_p^\sigma(D)
 \quad\Longleftrightarrow\quad
 \Lambda_pz\in\Omega_p(\widetilde D_\sigma)^-.
\]
Consequently, $\Lambda_p$ restricts to an isomorphism
\[
        \Lambda_p:\Omega_p^\sigma(D)
        \longrightarrow
        \Omega_p(\widetilde D_\sigma)^-.
\]
On $\Omega_p^\sigma(D)$, the boundary identity is
$\Lambda_{p-1}\partial_p^\sigma=\partial_p\Lambda_p$.
Thus $\Lambda_\bullet$ is an isomorphism from the signed GLMY chain complex of
$D$ to the anti-invariant GLMY subcomplex of the signed double cover.  Since
the deck involution commutes with the boundary and \(2\) is invertible over
\(\mathbb R\), the projections \(\frac12(1\pm\tau)\) split
$\Omega_\bullet(\widetilde D_\sigma)$ into its invariant and anti-invariant
subcomplexes.  The homology of the anti-invariant subcomplex is therefore
canonically identified with \(H_p(\widetilde D_\sigma)^-\).
\end{proof}

\begin{proposition}\label{prop:all-positive-reduction}
Let $(D,\sigma)$ be a signed digraph such that every arrow has sign \(+1\).  Then, for every \(p\),
\[
        H_p^\sigma(D)\cong H_p(D).
\]
More precisely, the signed double cover is the disjoint union \(D^+\sqcup D^-\), and the anti-invariant subcomplex is naturally identified with the ordinary GLMY chain complex of \(D\) by
\[
        u\longmapsto u^+-u^-.
\]
\end{proposition}

\begin{proof}
When all signs are positive, lifted arrows never change sheets, so \(\widetilde D_\sigma=D^+\sqcup D^-\).  The deck involution swaps the two copies.  Hence every anti-invariant \(p\)-chain is uniquely of the form \(u^+-u^-\), where \(u\) is an ordinary GLMY \(p\)-chain on \(D\).  The ordinary GLMY boundary acts in the same way on the two sheets:
\[
        \partial_p(u^+-u^-)=(\partial_pu)^+-(\partial_pu)^-.
\]
Thus the anti-invariant chain complex is canonically isomorphic to the ordinary GLMY chain complex of \(D\).  Passing to homology gives the assertion.
\end{proof}

\begin{proposition}[Signed digraph switching invariance]\label{prop:signed-digraph-switching-invariance}
Let $(D,\sigma)$ be a signed digraph and let $\eta:V(D)\to\Signs$.  Define the switched signing by
\[
        \sigma^\eta(ij)=\eta(i)\sigma(ij)\eta(j)
\]
for every arrow $i\to j$.  Then
\[
        H_p^\sigma(D)\cong H_p^{\sigma^\eta}(D)
\]
for every $p$.
\end{proposition}

\begin{proof}
To compare the signed double covers before and after switching, define the vertex map
\[
        \Phi_\eta:\widetilde D_{\sigma}\longrightarrow \widetilde D_{\sigma^\eta},
        \qquad
        \Phi_\eta(i^\varepsilon)=i^{\eta(i)\varepsilon}.
\]
If $i^\varepsilon\to j^{\varepsilon\sigma(ij)}$ is a lifted arrow in $\widetilde D_\sigma$, then its image is
\[
        i^{\eta(i)\varepsilon}
        \longrightarrow
        j^{\eta(j)\varepsilon\sigma(ij)}.
\]
In $\widetilde D_{\sigma^\eta}$, the arrow starting at $i^{\eta(i)\varepsilon}$ over $i\to j$ ends at
\[
        j^{\eta(i)\varepsilon\sigma^\eta(ij)}
        =
        j^{\eta(i)\varepsilon\eta(i)\sigma(ij)\eta(j)}
        =
        j^{\eta(j)\varepsilon\sigma(ij)}.
\]
Thus $\Phi_\eta$ is a digraph isomorphism.  Since $\Phi_\eta$ commutes
with the deck involutions, the induced homology isomorphism restricts to
\[
        H_p(\widetilde D_{\sigma})^-
        \cong
        H_p(\widetilde D_{\sigma^\eta})^-.
\]
The double-cover description of \cref{prop:double-cover-description} gives the claim.
\end{proof}
\begin{corollary}\label{cor:switching-balanced-reduction}
If \((D,\sigma)\) is switching-balanced, then
\[
        H_p^\sigma(D)\cong H_p(D)
        \qquad\text{for every }p.
\]
\end{corollary}

\begin{proof}
Choose a switching function \(\eta\) for which \(\sigma^\eta\) is all-positive.
Then \cref{prop:signed-digraph-switching-invariance,prop:all-positive-reduction}
together give
\[
        H_p^\sigma(D)
        \cong H_p^{\sigma^\eta}(D)
        \cong H_p(D).
\]
\end{proof}

The isomorphism in \cref{cor:switching-balanced-reduction} uses the chosen
switching function \(\eta\).  Two such functions differ by a constant sign
on each weak component, and the resulting identifications differ by that
sign on the component's homology.  For the all-positive signing, the
identification in \cref{prop:all-positive-reduction} is canonical.  By
\cref{prop:double-cover-splitting}, the switching-balanced case has the
geometric characterization that, on each weak component, the signed double
cover splits into two copies of the underlying digraph.

\begin{example}\label{ex:unbalanced-triangle}
Let \(D_{\mathrm{cyc}}\) be the directed cycle
\[
        0\to1,\qquad 1\to2,\qquad 2\to0,
\]
and equip it with the signing
\[
        \sigma(01)=\sigma(12)=+1,
        \qquad
        \sigma(20)=-1.
\]
Thus \((D_{\mathrm{cyc}},\sigma)\) is unbalanced.  Its signed double cover is
shown in \cref{fig:unbalanced-three-cycle-cover}.

\begin{figure}[htbp]
\centering
\begin{tikzpicture}[>=stealth,scale=1.0,
    vtx/.style={circle,fill=black,inner sep=1.8pt},
    posedge/.style={->,blue,thick},
    negedge/.style={->,red,thick},
    coveredge/.style={->,blue,thick}]
% signed base digraph: same vertex placement as Figure~\ref{fig:signed-triangle-lift-obstruction}
\begin{scope}[xshift=0cm]
  \node at (0,2.10) {$ (D_{\mathrm{cyc}},\sigma)$};
  \node[vtx] (b0) at (-1.7,-.25) {};
  \node[vtx] (b1) at (0,1.15) {};
  \node[vtx] (b2) at (1.7,-.25) {};
  \node[left=2pt]  at (b0) {$0$};
  \node[above=2pt] at (b1) {$1$};
  \node[right=2pt] at (b2) {$2$};
  \draw[posedge] (b0)--(b1);
  \draw[posedge] (b1)--(b2);
  \draw[negedge] (b2)--(b0);
\end{scope}
% signed double cover, with the same label placement as Figure~\ref{fig:signed-triangle-lift-obstruction}
\begin{scope}[xshift=7.0cm]
  \node at (0,2.10) {$\widetilde D_\sigma$};
  \node[vtx] (p0) at (-2.15,.55) {};
  \node[vtx] (p1) at (0,1.25) {};
  \node[vtx] (p2) at (2.15,.55) {};
  \node[vtx] (m0) at (-2.15,-.65) {};
  \node[vtx] (m1) at (0,-1.35) {};
  \node[vtx] (m2) at (2.15,-.65) {};
  \node[left=2pt]  at (p0) {$0^+$};
  \node[above=2pt] at (p1) {$1^+$};
  \node[right=2pt] at (p2) {$2^+$};
  \node[left=2pt]  at (m0) {$0^-$};
  \node[below=2pt] at (m1) {$1^-$};
  \node[right=2pt] at (m2) {$2^-$};
  \draw[coveredge] (p0)--(p1);
  \draw[coveredge] (p1)--(p2);
  \draw[coveredge] (m0)--(m1);
  \draw[coveredge] (m1)--(m2);
  \draw[negedge] (p2)--(m0);
  \draw[negedge] (m2)--(p0);
\end{scope}
\end{tikzpicture}
\caption{An unbalanced signed directed $3$-cycle and its signed double cover.
The vertex positions agree with Figure~\ref{fig:signed-triangle-lift-obstruction};
here the negative arrow is reversed, from $2$ to $0$.  Accordingly, its lifts
are $2^+\to0^-$ and $2^-\to0^+$.}
\label{fig:unbalanced-three-cycle-cover}
\end{figure}
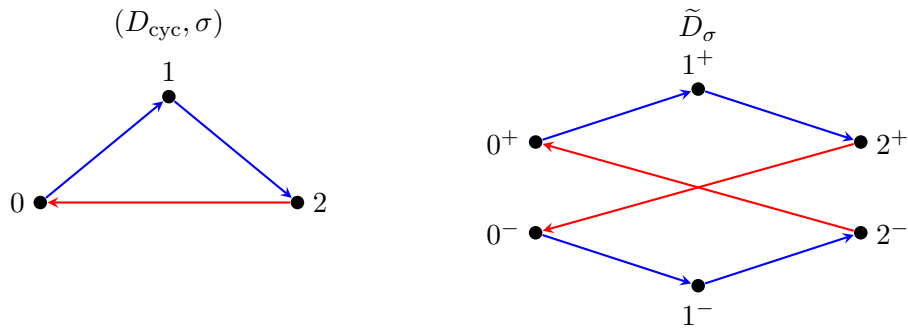

For the unbalanced signing \(\sigma\), the double cover in
\cref{fig:unbalanced-three-cycle-cover} is a directed $6$-cycle.  Hence its
ordinary first GLMY homology is one-dimensional, generated by the class of
the sum of the six directed edges.  The deck involution fixes this homology
class, so the anti-invariant part vanishes.  Therefore, by
\cref{prop:double-cover-description},
\[
        H_1^\sigma(D_{\mathrm{cyc}})=0.
\]
By contrast, for the all-positive signing \(\sigma_+\),
\cref{cor:switching-balanced-reduction} gives
\[
        H_1^{\sigma_+}(D_{\mathrm{cyc}})
        \cong H_1(D_{\mathrm{cyc}})
        \cong \mathbb R.
\]
Thus the first signed GLMY homology depends on the signing even when the
underlying directed graph is fixed.
\end{example}

\section{Canonical signed graph GLMY homology}\label{sec:canonical}

We now pass from signed digraphs to signed graphs by applying the preceding
construction to the canonical bidirected completion, in direct analogy with
the ordinary GLMY homology of an undirected graph.

\begin{definition}
Let $\Sigma=(G,\sigma)$ be a signed graph.  Its signed bidirected completion
$\overleftrightarrow\Sigma$ is obtained by replacing every edge
$uv\in E(G)$ by the two arrows $u\to v$ and $v\to u$, both carrying the
sign $\sigma(uv)$.  The canonical signed GLMY homology of $\Sigma$ is
\[
        \mathcal H_p^\sigma(\Sigma)
        :=H_p^\sigma(\overleftrightarrow\Sigma).
\]
By \cref{prop:double-cover-description},
\[
        \mathcal H_p^\sigma(\Sigma)
        \cong H_p(\widetilde{\overleftrightarrow{\Sigma}}_{\sigma})^-.
\]
\end{definition}

\begin{corollary}
\label{cor:canonical-switching-balanced}
Let \(\Sigma=(G,\sigma)\) be a signed graph and let
\(\eta:V(G)\to\Signs\).  Define
\[
        \sigma^\eta(uv)=\eta(u)\sigma(uv)\eta(v)
        \qquad (uv\in E(G)).
\]
Then, for every \(p\ge0\),
\[
        \mathcal H_p^\sigma(\Sigma)
        \cong
        \mathcal H_p^{\sigma^\eta}(\Sigma).
\]
If \(\Sigma\) is balanced, then
\[
        \mathcal H_p^\sigma(\Sigma)\cong H_p(G)
        \qquad\text{for every }p\ge0.
\]
\end{corollary}

\begin{proof}
Switching \(\Sigma\) by \(\eta\) switches both arrows over each edge of
\(\overleftrightarrow\Sigma\) by the same vertex function.  Hence the
first assertion is \cref{prop:signed-digraph-switching-invariance} applied
to the bidirected completion.  If \(\Sigma\) is balanced, Harary's
criterion gives a switching for which all edge signs are positive.  The
second assertion then follows from the first assertion and
\cref{prop:all-positive-reduction}, together with
\(H_p(G)=H_p(\overleftrightarrow G)\).
\end{proof}

We now compare the degree-zero canonical signed GLMY group with the usual
signed graph Laplacian.  For a signed graph \(\Sigma=(G,\sigma)\), let
\(A_\sigma\) be the signed adjacency matrix,
\[
 (A_\sigma)_{uv}
 =\begin{cases}
   \sigma(uv),&uv\in E(G),\\
   0,&uv\notin E(G),
  \end{cases}
\]
and let \(D_G\) be the diagonal degree matrix of \(G\).  The usual signed
graph Laplacian is
\[
        L^\sigma(\Sigma):=D_G-A_\sigma.
\]
Equivalently, for
\[
        h=\sum_{v\in V(G)}h(v)e_v,
\]
one has
\[
        (L^\sigma h)(u)
        =\deg(u)h(u)-\sum_{v\sim u}\sigma(uv)h(v).
\]
We shall prove that the degree-zero canonical signed GLMY homology is
naturally identified with the kernel of this operator:
\[
        \mathcal H_0^\sigma(\Sigma)\cong\ker L^\sigma(\Sigma).
\]

To relate \(L^\sigma\) to the signed GLMY boundary, identify
\(\Omega_0^\sigma(\overleftrightarrow\Sigma)\) with
\[
        C_0(\Sigma)
        =\operatorname{span}_{\mathbb R}\{e_v:v\in V(G)\}
\]
by \([e_{v^+}]\mapsto e_v\), and equip \(C_0(\Sigma)\) with the standard
inner product
\[
        \Big\langle\sum_v h(v)e_v,\sum_v k(v)e_v\Big\rangle_0
        :=\sum_{v\in V(G)}h(v)k(v).
\]
Choose one temporary orientation \(u\to v\) for each edge \(uv\in E(G)\),
and let
\[
        C_1(\Sigma)
        =\operatorname{span}_{\mathbb R}
        \{e_{uv}:u\to v\text{ is the chosen orientation of }uv\in E(G)\},
\]
with the basis \(\{e_{uv}\}\) orthonormal.  Define the signed incidence map
\[
        B_\sigma:C_1(\Sigma)\longrightarrow C_0(\Sigma),
        \qquad
        B_\sigma e_{uv}=\sigma(uv)e_v-e_u.
\]
Let \(B_\sigma^*:C_0(\Sigma)\to C_1(\Sigma)\) be the adjoint.  A direct
calculation gives the standard factorization
\[
        L^\sigma(\Sigma)=B_\sigma B_\sigma^*.
\]
Reversing the temporary orientation of one edge multiplies the
corresponding column of \(B_\sigma\) by \(-\sigma(uv)\), so the product
\(B_\sigma B_\sigma^*\) is independent of all temporary orientations.

By \cref{prop:degree-one-incidence}, the two arrows of
\(\overleftrightarrow\Sigma\) lying over an edge \(uv\), temporarily
oriented as \(u\to v\), satisfy
\[
\begin{aligned}
 \partial_1^\sigma
 \bigl[e_{u^+v^{\sigma(uv)}}\bigr]
 &=\sigma(uv)e_v-e_u
  =B_\sigma e_{uv},\\
 \partial_1^\sigma
 \bigl[e_{v^+u^{\sigma(uv)}}\bigr]
 &=\sigma(uv)e_u-e_v
  =-\sigma(uv)B_\sigma e_{uv}.
\end{aligned}
\]
Consequently,
\[
        \operatorname{im}\bigl(
        \partial_1^\sigma:\Omega_1^\sigma(\overleftrightarrow\Sigma)
        \to\Omega_0^\sigma(\overleftrightarrow\Sigma)\bigr)
        =\operatorname{im}B_\sigma.
\]
Thus the degree-zero signed GLMY group is controlled by the same incidence
operator as the usual signed graph Laplacian.  The precise identification is
the following.

\begin{corollary}[Zero-dimensional signed GLMY homology]\label{cor:canonical-H0}
For every signed graph $\Sigma=(G,\sigma)$,
\[
        \mathcal H_0^\sigma(\Sigma)
        \cong \ker L^\sigma(\Sigma).
\]
Consequently, the dimension of $\mathcal H_0^\sigma(\Sigma)$ is the number of balanced connected components of $\Sigma$.
\end{corollary}

\begin{proof}
Since the degree-zero boundary is zero and the degree-one boundary has the
same image as $B_\sigma$,
\[
        \mathcal H_0^\sigma(\Sigma)
        =C_0(\Sigma)/\operatorname{im}B_\sigma
        =\operatorname{coker}B_\sigma.
\]
By the definition of the adjoint,
\[
        (\operatorname{im}B_\sigma)^\perp
        =\ker B_\sigma^*.
\]
Indeed, for $h\in C_0(\Sigma)$,
\[
\begin{aligned}
 h\in(\operatorname{im}B_\sigma)^\perp
 &\Longleftrightarrow
 \langle B_\sigma c,h\rangle_0=0
 \quad\text{for every }c\in C_1(\Sigma)\\
 &\Longleftrightarrow
 \langle c,B_\sigma^*h\rangle_1=0
 \quad\text{for every }c\in C_1(\Sigma)\\
 &\Longleftrightarrow B_\sigma^*h=0.
\end{aligned}
\]
Since $C_0(\Sigma)$ is finite-dimensional, we therefore have the orthogonal
decomposition
\[
        C_0(\Sigma)
        =\operatorname{im}B_\sigma\oplus\ker B_\sigma^*.
\]
Orthogonal projection onto $\ker B_\sigma^*$ consequently induces an
isomorphism
\[
        \operatorname{coker}B_\sigma
        =C_0(\Sigma)/\operatorname{im}B_\sigma
        \cong\ker B_\sigma^*.
\]
For a zero-chain
\[
        h=\sum_{v\in V(G)}h(v)e_v\in C_0(\Sigma),
\]
we have
\[
        \langle h,L^\sigma(\Sigma)h\rangle_0
        =\langle B_\sigma^*h,B_\sigma^*h\rangle_1
        =\lVert B_\sigma^*h\rVert_1^2.
\]
Hence
\[
        \ker L^\sigma(\Sigma)=\ker B_\sigma^*,
\]
and therefore
\[
        \mathcal H_0^\sigma(\Sigma)\cong\ker L^\sigma(\Sigma).
\]
This proves the isomorphism
$\mathcal H_0^\sigma(\Sigma)\cong\ker L^\sigma(\Sigma)$.  It remains to
determine the dimension of $\ker L^\sigma(\Sigma)$.

If the edge \(uv\) is temporarily oriented as \(u\to v\), then the
coefficient of \(e_{uv}\) in \(B_\sigma^*h\) is
\[
        \sigma(uv)h(v)-h(u).
\]
Thus every zero-chain in \(\ker L^\sigma(\Sigma)\) satisfies
\[
        h(v)=\sigma(uv)h(u)
\]
along every edge \(uv\).  On each connected component, these equations
determine \(h\) from its value at one vertex.  A nonzero initial value
extends consistently precisely when every cycle has sign \(+1\).
Hence a balanced component contributes one dimension, while an
unbalanced component contributes zero dimensions.
\end{proof}

\begin{example}\label{ex:canonical-sign-sensitive-bowtie}
Let $B$ be the bowtie graph consisting of two triangles meeting at the vertex
$0$:
\[
        E(B)=\{01,12,20,03,34,40\}.
\]
For the all-positive signing \(\sigma_+\), both ordinary triangle classes are killed by the
canonical bidirected GLMY two-path fillings, and therefore
\[
        \mathcal H_0^{\sigma_+}(B)\cong\mathbb R,
        \qquad
        \mathcal H_1^{\sigma_+}(B)=0.
\]
For the signing
\[
        \sigma(01)=\sigma(03)=-1,
        \qquad
        \sigma(12)=\sigma(20)=\sigma(34)=\sigma(40)=+1,
\]
the cycles $0120$ and $0340$ are unbalanced, so
\(\mathcal H_0^\sigma(B)=0\).

\begin{figure}[htbp]
\centering
\resizebox{0.98\textwidth}{!}{%
\begin{tikzpicture}[
  >=stealth,
  dot/.style={circle,fill=black,inner sep=1.8pt},
  vlabel/.style={fill=white,inner sep=0.7pt,font=\small},
  sgpos/.style={blue,thick},
  sgneg/.style={red,thick}
]
  % all-positive bowtie
  \begin{scope}[xshift=0cm]
    \node[dot] (a0) at (0,0) {};
    \node[dot] (a1) at (-1.25,0.95) {};
    \node[dot] (a2) at (-1.25,-0.95) {};
    \node[dot] (a3) at (1.25,0.95) {};
    \node[dot] (a4) at (1.25,-0.95) {};
    \draw[sgpos] (a0)--(a1);
    \draw[sgpos] (a1)--(a2);
    \draw[sgpos] (a2)--(a0);
    \draw[sgpos] (a0)--(a3);
    \draw[sgpos] (a3)--(a4);
    \draw[sgpos] (a4)--(a0);
    \node[vlabel,above left=2pt] at (a1) {$1$};
    \node[vlabel,below left=2pt] at (a2) {$2$};
    \node[vlabel,above=3pt] at (a0) {$0$};
    \node[vlabel,above right=2pt] at (a3) {$3$};
    \node[vlabel,below right=2pt] at (a4) {$4$};
    \node[font=\small] at (0,-1.55) {all-positive signing};
    \node[font=\small] at (0,1.55) {$\mathcal H_1^{\sigma_+}(B)=0$};
  \end{scope}

  % unbalanced bowtie
  \begin{scope}[xshift=5.4cm]
    \node[dot] (b0) at (0,0) {};
    \node[dot] (b1) at (-1.25,0.95) {};
    \node[dot] (b2) at (-1.25,-0.95) {};
    \node[dot] (b3) at (1.25,0.95) {};
    \node[dot] (b4) at (1.25,-0.95) {};
    \draw[sgneg] (b0)--(b1);
    \draw[sgpos] (b1)--(b2);
    \draw[sgpos] (b2)--(b0);
    \draw[sgneg] (b0)--(b3);
    \draw[sgpos] (b3)--(b4);
    \draw[sgpos] (b4)--(b0);
    \node[vlabel,above left=2pt] at (b1) {$1$};
    \node[vlabel,below left=2pt] at (b2) {$2$};
    \node[vlabel,above=3pt] at (b0) {$0$};
    \node[vlabel,above right=2pt] at (b3) {$3$};
    \node[vlabel,below right=2pt] at (b4) {$4$};
    \node[font=\small] at (0,-1.55) {two unbalanced triangles};
    \node[font=\small] at (0,1.55) {$\mathcal H_1^\sigma(B)\cong\mathbb R$};
  \end{scope}

  % underlying undirected graph of the signed double cover of the second signing
  \begin{scope}[xshift=11.5cm]
    \node[font=\small] at (0,2.25) {signed double cover};
    \node[dot] (p0) at (0,0.75) {};
    \node[dot] (m0) at (0,-0.75) {};

    % positive-sheet vertices (upper branches)
    \node[dot] (p2) at (-2.15,1.55) {};
    \node[dot] (p1) at (-2.15,0.45) {};
    \node[dot] (p4) at (2.15,1.55) {};
    \node[dot] (p3) at (2.15,0.45) {};

    % negative-sheet vertices (lower branches)
    \node[dot] (m1) at (-2.15,-0.45) {};
    \node[dot] (m2) at (-2.15,-1.55) {};
    \node[dot] (m3) at (2.15,-0.45) {};
    \node[dot] (m4) at (2.15,-1.55) {};

    % four internally disjoint length-three paths from 0+ to 0-
    \draw[sgpos] (p0)--(p2);
    \draw[sgpos] (p2)--(p1);
    \draw[sgneg] (p1)--(m0);

    \draw[sgpos] (p0)--(p4);
    \draw[sgpos] (p4)--(p3);
    \draw[sgneg] (p3)--(m0);

    \draw[sgneg] (p0)--(m1);
    \draw[sgpos] (m1)--(m2);
    \draw[sgpos] (m2)--(m0);

    \draw[sgneg] (p0)--(m3);
    \draw[sgpos] (m3)--(m4);
    \draw[sgpos] (m4)--(m0);

    \node[vlabel,above=2pt] at (p0) {$0^+$};
    \node[vlabel,below=2pt] at (m0) {$0^-$};
    \node[vlabel,above=2pt] at (p2) {$2^+$};
    \node[vlabel,left=3pt] at (p1) {$1^+$};
    \node[vlabel,above=2pt] at (p4) {$4^+$};
    \node[vlabel,right=3pt] at (p3) {$3^+$};
    \node[vlabel,left=3pt] at (m1) {$1^-$};
    \node[vlabel,below=2pt] at (m2) {$2^-$};
    \node[vlabel,right=3pt] at (m3) {$3^-$};
    \node[vlabel,below=2pt] at (m4) {$4^-$};
  \end{scope}
\end{tikzpicture}%
}
\caption{Two signings of the bowtie graph and the signed double cover of the
second signing.  Blue solid edges are positive and red solid edges are
negative in the base graphs.  In the right-hand diagram, colors record the
sign of the projected base edge; the cover itself is unsigned.  Its
underlying undirected graph consists of four internally disjoint paths of
length three joining $0^+$ to $0^-$.}
\label{fig:bowtie-sign-sensitive}
\end{figure}
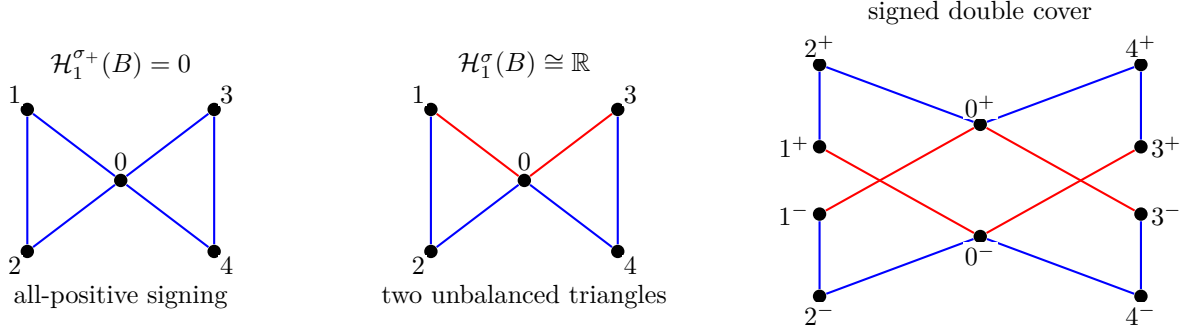

The right-hand panel of \cref{fig:bowtie-sign-sensitive} shows that the
signed double cover consists of four internally disjoint paths of length three
joining \(0^+\) to \(0^-\), with no triangle or quadrilateral fillings.
Every non-backtracking allowed two-path in this cover has a non-allowed
shortcut face that occurs in no other allowed two-path.  Hence every
boundary-preserving two-chain is a linear combination of backtracks, whose
boundaries impose only the usual reversal relations on one-chains.  The deck
anti-invariant first GLMY homology is therefore one-dimensional.  To make a
generator visible in the diagram,
follow the six-cycle
\[
        0^+\to1^-\to2^-\to0^-\to4^-\to3^-\to0^+.
\]
This six-cycle determines the signed one-chain
\[
\begin{aligned}
 z_\sigma={}&[e_{0^+1^-}]+[e_{1^-2^-}]+[e_{2^-0^-}]\\
             &+[e_{0^-4^-}]+[e_{4^-3^-}]+[e_{3^-0^+}]
 \in \Omega_1^\sigma(\overleftrightarrow B).
\end{aligned}
\]
The boundary telescopes around this six-cycle, so
\[
        \partial_1^\sigma z_\sigma=0.
\]
Because every degree-two boundary is generated by backtrack relations,
this six-cycle is not a boundary.  Hence the homology class
\[
        [z_\sigma]\in\mathcal H_1^\sigma(B)
\]
is nonzero and generates the one-dimensional group.  Therefore
\[
        \mathcal H_1^\sigma(B)\cong\mathbb R.
\]
Thus the first canonical signed GLMY homology distinguishes the
all-positive signing from the signing with negative edges $01$ and $03$.
\end{example}

\section{Functoriality of signed GLMY homology}\label{sec:functoriality}

We now identify the natural morphisms compatible with the signed double-cover
construction and prove that signed GLMY homology is functorial with respect
to them.

\subsection{Signed digraph weak morphisms}

\begin{definition}\label{def:signed-weak-morphism}
Let $(D,\sigma)$ and $(D',\sigma')$ be signed digraphs.  A signed weak morphism
\[
        (f,\eta):(D,\sigma)\longrightarrow(D',\sigma')
\]
consists of a vertex map $f:V(D)\to V(D')$ and a switching function $\eta:V(D)\to\Signs$ such that for every arrow $i\to j$ of $D$ the following conditions hold.
\begin{enumerate}[label=\textup{(\roman*)}]
    \item If $f(i)\ne f(j)$, then $f(i)\to f(j)$ is an arrow of $D'$ and
    \[
            \sigma'(f(i)f(j))=\eta(i)\sigma(ij)\eta(j).
    \]
    \item If $f(i)=f(j)$, then
    \[
            \sigma(ij)=\eta(i)\eta(j).
    \]
\end{enumerate}
Equivalently, $f$ sends each arrow of $(D,\sigma^\eta)$ either to an
arrow of the same sign or, when $\sigma^\eta(ij)=+1$, to a vertex.
The switching function $\eta$ is part of the morphism data.
\end{definition}

\begin{example}
\label{ex:signed-weak-morphism-unbalanced-triangle}
Let $D$ have arrows $0\to1$, $1\to2$ and $2\to0$, with
\[
        \sigma(01)=-1,\qquad \sigma(12)=\sigma(20)=+1.
\]
The cycle $0\to1\to2\to0$ has sign $-1$.  Let $D'$ have vertices
$x,y$ and the
two arrows $x\to y$ and $y\to x$, with
\[
        \sigma'(xy)=-1,\qquad \sigma'(yx)=+1.
\]
The antiparallel arrows carry independent signs, as permitted by
\cref{def:signed-digraph}.  Set
\[
\begin{aligned}
        f(0)=f(1)&=x, & f(2)&=y,\\
        \eta(0)=\eta(2)&=+1, & \eta(1)&=-1.
\end{aligned}
\]
The switched signs are
\[
\begin{aligned}
        \sigma^\eta(01)&=(+1)(-1)(-1)=+1,\\
        \sigma^\eta(12)&=(-1)(+1)(+1)=-1,\\
        \sigma^\eta(20)&=(+1)(+1)(+1)=+1.
\end{aligned}
\]
Thus $f$ contracts the positive arrow $0\to1$ of $(D,\sigma^\eta)$
and sends $1\to2$ and $2\to0$ to $x\to y$ and $y\to x$ with the
same signs.  Hence $(f,\eta)$ is a signed weak morphism.  The switched
cycle still has sign $-1$.  The two stages, switching and contraction,
are shown in \cref{fig:signed-weak-morphism-example}.

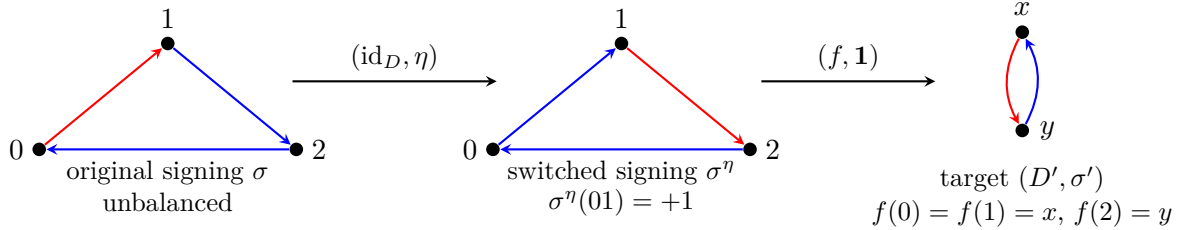
\begin{figure}[htbp]
\centering
\begin{tikzpicture}[>=stealth,scale=1.0,
        ppos/.style={->,blue,thick},
        nneg/.style={->,red,thick},
        vtx/.style={circle,fill=black,inner sep=1.8pt}]

% left: original triangle
\begin{scope}[xshift=0cm]
    \node[vtx] (a0) at (-1.7,-.25) {};
    \node[vtx] (a1) at (0,1.15) {};
    \node[vtx] (a2) at (1.7,-.25) {};
    \node[left=2pt]  at (a0) {$0$};
    \node[above=2pt] at (a1) {$1$};
    \node[right=2pt] at (a2) {$2$};
    \draw[nneg] (a0) -- (a1);
    \draw[ppos] (a1) -- (a2);
    \draw[ppos] (a2) -- (a0);
    \node[font=\small] at (0,-0.55) {original signing $\sigma$};
    \node[font=\small] at (0,-0.95) {unbalanced};
\end{scope}

% arrow to switched
\draw[->,thick] (1.65,.65) -- node[above,font=\small] {$(\operatorname{id}_D,\eta)$} (4.35,.65);

% middle: switched triangle
\begin{scope}[xshift=6.0cm]
    \node[vtx] (b0) at (-1.7,-.25) {};
    \node[vtx] (b1) at (0,1.15) {};
    \node[vtx] (b2) at (1.7,-.25) {};
    \node[left=2pt]  at (b0) {$0$};
    \node[above=2pt] at (b1) {$1$};
    \node[right=2pt] at (b2) {$2$};
    \draw[ppos] (b0) -- (b1);
    \draw[nneg] (b1) -- (b2);
    \draw[ppos] (b2) -- (b0);
    \node[font=\small] at (0,-0.55) {switched signing $\sigma^\eta$};
    \node[font=\small] at (0,-0.95) {$\sigma^\eta(01)=+1$};
\end{scope}

% arrow to target
\draw[->,thick] (7.85,.65) -- node[above,font=\small] {$(f,\mathbf1)$} (10.15,.65);

% right: target digraph
\begin{scope}[xshift=11.3cm]
    \node[vtx,label=above:{$x$}] (c0) at (0,1.3) {};
    \node[vtx,label=right:{$y$}] (c1) at (0,0) {};
    \draw[ppos,bend right=28] (c1) to (c0);
    \draw[nneg,bend right=28] (c0) to (c1);
    \node[font=\small] at (0,-0.65) {target $(D',\sigma')$};
    \node[font=\small,align=center] at (0,-1.1) {$f(0)=f(1)=x$, $f(2)=y$};
\end{scope}
\end{tikzpicture}
\caption{The morphism of \cref{ex:signed-weak-morphism-unbalanced-triangle}
as a switching followed by a contraction.  Switching changes the negative
arrow from $0\to1$ to $1\to2$; the positive arrow $0\to1$ is then
contracted.  Blue arrows are positive and red arrows are negative.}
\label{fig:signed-weak-morphism-example}
\end{figure}
\FloatBarrier
\end{example}

For composable signed weak morphisms
$(f,\eta):(D,\sigma)\to(D',\sigma')$ and
$(g,\theta):(D',\sigma')\to(D'',\sigma'')$, define
\[
        (g,\theta)\circ(f,\eta)
        =
        (g\circ f,\zeta),
        \qquad
        \zeta(i)=\theta(f(i))\eta(i).
\]
The composite satisfies \cref{def:signed-weak-morphism}.  If an arrow
\(i\to j\) is retained by both maps, its sign condition follows by
multiplying the two sign identities.
If \(f\) contracts \(i\to j\), then \(\sigma(ij)=\eta(i)\eta(j)\), and the
formula for \(\zeta\) gives the contraction condition for \(g\circ f\).  If
\(f(i)\ne f(j)\) but \(g\) contracts the arrow \(f(i)\to f(j)\), then
\[
        \eta(i)\sigma(ij)\eta(j)
        =
        \sigma'(f(i)f(j))
        =
        \theta(f(i))\theta(f(j)),
\]
which again gives \(\sigma(ij)=\zeta(i)\zeta(j)\).
The identity is \((\operatorname{id}_D,\mathbf1)\), where
\(\mathbf1(v)=+1\) for every vertex.  For three successive morphisms \((f,\eta)\), \((g,\theta)\) and
\((h,\psi)\), the switching factor at \(v\) is
\(\psi(g(f(v)))\theta(f(v))\eta(v)\), independently of how the
composition is bracketed.  Thus composition is associative.

\begin{lemma}\label{lem:weak-morphism-lift}
Let \((f,\eta):(D,\sigma)\to(D',\sigma')\) be a signed weak morphism,
and let \(\tau,\tau'\) and \(\pi,\pi'\) be the deck involutions and
projections of the respective covers.  The pair $(f,\eta)$ induces the weak digraph morphism
\[
        \widetilde f:\widetilde D_\sigma\longrightarrow
        \widetilde D'_{\sigma'},
        \qquad
        \widetilde f(i^\varepsilon)=f(i)^{\eta(i)\varepsilon}.
\]
The lift depends on the specified pair \((f,\eta)\) and satisfies
\[
        \pi'\circ\widetilde f=f\circ\pi,
        \qquad \widetilde f\circ\tau=\tau'\circ\widetilde f.
\]
The sheet label over \(i\) is preserved when \(\eta(i)=+1\) and
reversed when \(\eta(i)=-1\).
\end{lemma}

\begin{proof}
Let $i^\varepsilon\to j^{\varepsilon\sigma(ij)}$ be an arrow of $\widetilde D_\sigma$ lying over $i\to j$.  If $f(i)\ne f(j)$, then $f(i)\to f(j)$ is an arrow of $D'$ and
\[
        \eta(j)\varepsilon\sigma(ij)
        =
        \eta(i)\varepsilon\sigma'(f(i)f(j)).
\]
Thus $\widetilde f$ sends the lifted arrow to the lifted arrow over $f(i)\to f(j)$.  If $f(i)=f(j)$, then $\sigma(ij)=\eta(i)\eta(j)$, so
\[
        f(j)^{\eta(j)\varepsilon\sigma(ij)}
        =
        f(i)^{\eta(i)\varepsilon}.
\]
Thus the lifted arrow is contracted to one vertex.  Hence $\widetilde f$ is a weak digraph morphism.  The identity
\[
        \widetilde f(\tau(i^\varepsilon))
        =
        \tau'(\widetilde f(i^\varepsilon))
\]
and the projection identity follow from the definition of
$\widetilde f$.
\end{proof}

For the morphism in \cref{ex:signed-weak-morphism-unbalanced-triangle},
the lift is given, for $\varepsilon\in\Signs$, by
\[
 \widetilde f(0^\varepsilon)=x^\varepsilon,\qquad
 \widetilde f(1^\varepsilon)=x^{-\varepsilon},\qquad
 \widetilde f(2^\varepsilon)=y^\varepsilon.
\]

\begin{definition}\label{def:induced-signed-chain-map}
Let $(f,\eta):(D,\sigma)\to(D',\sigma')$ be a signed weak morphism, and let
$\widetilde f$ be its lift.  For every regular lifted path
\(\widetilde\gamma=i_0^{\varepsilon_0}\cdots i_p^{\varepsilon_p}\), its
vertexwise image is
\[
 \widetilde f(\widetilde\gamma)
 =f(i_0)^{\eta(i_0)\varepsilon_0}\cdots
  f(i_p)^{\eta(i_p)\varepsilon_p}.
\]
Define
\[
        (f,\eta)_\#\bigl([e_{\widetilde\gamma}]\bigr)
        =
        \begin{cases}
        [e_{\widetilde f(\widetilde\gamma)}],&
        \widetilde f(\widetilde\gamma)\text{ is regular},\\
        0,&
        \widetilde f(\widetilde\gamma)\text{ is not regular}.
        \end{cases}
\]
The identity $\widetilde f\tau=\tau'\widetilde f$ ensures that the
vertexwise path map preserves the deck relations and hence defines
$(f,\eta)_\#$ on $R_p^\sigma(D)$.  By
\cref{lem:weak-morphism-lift}, an allowed lifted path has either an
allowed image or a nonregular image; a nonregular image represents zero in
the target regular path module.  Thus $(f,\eta)_\#$ restricts to
\[
        (f,\eta)_\#:\mathcal A_p^\sigma(D)
        \longrightarrow\mathcal A_p^{\sigma'}(D').
\]
\end{definition}

For an allowed base path \(\gamma=i_0\cdots i_p\), write
\(f(\gamma)=f(i_0)\cdots f(i_p)\).  For generators represented by lifts starting in the positive sheet, the
induced map takes the form
\[
 (f,\eta)_\#[e_{\widetilde\gamma^+}]
 =\begin{cases}
   \eta(i_0)[e_{\widetilde{f(\gamma)}^{\,+}}],
       &f(\gamma)\text{ is regular},\\
   0,   &f(\gamma)\text{ is nonregular}.
  \end{cases}
\]
The image lift begins in sheet $\eta(i_0)$.  If $\eta(i_0)=-1$, the deck
relation converts the negative-sheet representative into minus the
positive-sheet representative, which accounts for the coefficient
$\eta(i_0)$.  When consecutive base images coincide, the collapse condition
identifies the corresponding lifted vertices.

In \cref{ex:signed-weak-morphism-unbalanced-triangle}, the three arrow
generators whose initial vertices lie in the positive sheet satisfy
\[
\begin{aligned}
 (f,\eta)_\#[e_{0^+1^-}]&=0,\\
 (f,\eta)_\#[e_{1^+2^+}]&=[e_{x^-y^+}]=-[e_{x^+y^-}],\\
 (f,\eta)_\#[e_{2^+0^+}]&=[e_{y^+x^+}].
\end{aligned}
\]
The first line is the contracted-arrow case.  In the second line,
$\eta(1)=-1$ places the image lift in the negative sheet, and the deck
relation gives $[e_{x^-y^+}]=-[e_{x^+y^-}]$.

\begin{lemma}\label{lem:lift-chain-map-compatibility}
For every $p$,
\[
        \Lambda'_p (f,\eta)_\#
        =
        \widetilde f_\# \Lambda_p,
\]
where $\Lambda_p$ and $\Lambda'_p$ are the isomorphisms of
\cref{prop:anti-invariant-realization}.
\end{lemma}

\begin{proof}
The identity holds on the ambient regular path modules.  Let
$\widetilde\gamma$ be a regular lifted path.  If
$\widetilde f(\widetilde\gamma)$ is regular, then
\[
\begin{aligned}
\Lambda'_p\bigl((f,\eta)_\#([e_{\widetilde\gamma}])\bigr)
&=e_{\widetilde f(\widetilde\gamma)}
  -e_{\tau'\widetilde f(\widetilde\gamma)}\\
&=\widetilde f_\#
  \bigl(e_{\widetilde\gamma}-e_{\tau\widetilde\gamma}\bigr)
 =\widetilde f_\#\bigl(\Lambda_p([e_{\widetilde\gamma}])\bigr),
\end{aligned}
\]
where the second equality uses
$\widetilde f\tau=\tau'\widetilde f$.  If
$\widetilde f(\widetilde\gamma)$ is not regular, both sides are zero by the
regular-path convention.
\end{proof}

\begin{theorem}[Functoriality for signed digraphs]\label{thm:signed-digraph-functoriality}
A signed weak morphism
\[
        (f,\eta):(D,\sigma)\longrightarrow(D',\sigma')
\]
induces a chain map
\[
        (f,\eta)_\#:
        \Omega_\bullet^\sigma(D)
        \longrightarrow
        \Omega_\bullet^{\sigma'}(D'),
        \qquad
        \partial_*^{\sigma'}(f,\eta)_\#
        =(f,\eta)_\#\partial_*^\sigma ,
\]
and therefore linear maps
\[
        (f,\eta)_*:H_p^\sigma(D)\longrightarrow H_p^{\sigma'}(D')
\]
for all \(p\ge0\).  The induced map $(f,\eta)_*$ sends the homology
class \([z]\) to \([(f,\eta)_\#z]\), and the induced maps satisfy
\[
 (\operatorname{id}_D,\mathbf1)_*=\operatorname{id},\qquad
 \bigl((g,\theta)\circ(f,\eta)\bigr)_*
       =(g,\theta)_*\circ(f,\eta)_*.
\]
Thus \((D,\sigma)\mapsto H_p^\sigma(D)\) is a functor from signed
digraphs with signed weak morphisms to real vector spaces.
\end{theorem}

\begin{proof}
By \cref{lem:weak-morphism-lift}, $\widetilde f$ is a weak morphism of
ordinary digraphs.  Ordinary GLMY functoriality
\cite{GLMY2014Homotopy} gives a chain map
$\widetilde f_\#:\Omega_\bullet(\widetilde D_\sigma)\to
\Omega_\bullet(\widetilde D'_{\sigma'})$.  Deck equivariance implies that
$\widetilde f_\#$ maps $\Omega_\bullet(\widetilde D_\sigma)^-$ into
$\Omega_\bullet(\widetilde D'_{\sigma'})^-$.
By \cref{lem:lift-chain-map-compatibility,prop:double-cover-description},
the restriction to the anti-invariant subcomplex corresponds to the
signed chain map
$(f,\eta)_\#=(\Lambda'_\bullet)^{-1}\widetilde f_\#\Lambda_\bullet$.
The lifted maps preserve identities and compose according to
$\zeta(i)=\theta(f(i))\eta(i)$, so identities and composition are
preserved on signed chains and on homology.
\end{proof}

The switching function is essential data in a signed weak morphism.  Indeed,
if $(f,\eta)$ is a signed weak morphism, then so is $(f,-\eta)$, since the
two minus signs cancel in each edge condition.  However, the corresponding
lifts satisfy
\[
        \widetilde f_{-\eta}=\tau'\circ\widetilde f_{\eta}.
\]
Since $\tau'$ acts as $-\operatorname{id}$ on the anti-invariant
subcomplex of the target double cover, the induced homology maps satisfy
\[
        (f,-\eta)_*=-(f,\eta)_*.
\]
Thus the underlying vertex map $f$ alone does not determine the induced map
on signed GLMY homology.

\subsection{Signed graph weak morphisms}

\begin{definition}\label{def:signed-graph-weak-morphism}
Let $\Sigma=(G,\sigma)$ and $\Sigma'=(G',\sigma')$ be signed graphs.  A signed weak morphism
\[
        (f,\eta):\Sigma\longrightarrow\Sigma'
\]
consists of a vertex map $f:V(G)\to V(G')$ and a switching function $\eta:V(G)\to\Signs$ such that for every edge $uv\in E(G)$:
\begin{enumerate}[label=\textup{(\roman*)}]
    \item if $f(u)\ne f(v)$, then $f(u)f(v)\in E(G')$ and
    \[
            \sigma'(f(u)f(v))=\eta(u)\sigma(uv)\eta(v);
    \]
    \item if $f(u)=f(v)$, then
    \[
            \sigma(uv)=\eta(u)\eta(v).
    \]
\end{enumerate}
\end{definition}

The edge conditions in \cref{def:signed-graph-weak-morphism} are
equivalent to the signed weak digraph-morphism conditions on both arrows
of the bidirected completion.  In particular, every contracted edge
satisfies $\sigma^\eta(uv)=+1$.

\begin{theorem}[Functoriality for signed graphs]\label{thm:canonical-functoriality}
Every signed weak morphism
\[
        (f,\eta):\Sigma\longrightarrow\Sigma'
\]
induces linear maps
\[
        (f,\eta)_*:
        \mathcal H_p^\sigma(\Sigma)
        \longrightarrow
        \mathcal H_p^{\sigma'}(\Sigma')
\]
for all \(p\ge0\), and these maps are compatible with identities and
composition.  Hence canonical signed GLMY homology is functorial on the
category of signed graphs and signed weak morphisms.
\end{theorem}

\begin{proof}
The bidirected completion sends a signed weak graph morphism to a signed weak digraph morphism
\[
        \overleftrightarrow\Sigma
        \longrightarrow
        \overleftrightarrow{\Sigma'}.
\]
The conclusion follows from \cref{thm:signed-digraph-functoriality}.
\end{proof}

\subsection{A suspension example}\label{sec:functoriality-example}

\begin{example}\label{ex:suspension-functor}
Let $D_L$ have vertices
\[
        \{\ast,1,2,\ldots,8,9\}
\]
and arrows
\[
\begin{gathered}
        \ast\to i\quad (1\le i\le 8),
        \qquad
        9\to 2,4,6,8,\\
        2\to1,\quad 2\to3,\quad
        4\to3,\quad 4\to5,\quad
        6\to5,\quad 6\to7,\quad
        8\to7,\quad 8\to1.
\end{gathered}
\]
Let $D_R$ be the suspension of the directed $4$-cycle
\[
        0\to1\to2\to3\to0,
\]
with suspension vertices $v,w$ and arrows
\[
        v,w\to0,1,2,3.
\]
Define a vertex map $f:D_L\to D_R$ by
\[
        f(\ast)=v,\qquad f(9)=w,
\]
and
\[
        f(1)=f(2)=0,\quad
        f(3)=f(4)=1,\quad
        f(5)=f(6)=2,\quad
        f(7)=f(8)=3.
\]
Thus the arrows $2\to1$, $4\to3$, $6\to5$, and $8\to7$ are contracted,
while the remaining four arrows $2\to3$, $4\to5$, $6\to7$, and
$8\to1$ map to the directed $4$-cycle.

Given any signing $\sigma_R$ on $D_R$, define a signing on $D_L$ by
\[
        \sigma_L(xy)=
        \begin{cases}
        1,& f(x)=f(y),\\
        \sigma_R(f(x)f(y)),& f(x)\ne f(y).
        \end{cases}
\]
With \(\eta\equiv1\), the pair
\((f,\eta):(D_L,\sigma_L)\to(D_R,\sigma_R)\) is a signed weak morphism.
Here \(\beta_p^{\sigma_L}(D_L)\) and \(\beta_p^{\sigma_R}(D_R)\)
denote the \(p\)-th signed Betti numbers of the left-hand digraph
\(D_L\) with signing \(\sigma_L\) and the right-hand digraph
\(D_R\) with signing \(\sigma_R\), respectively.

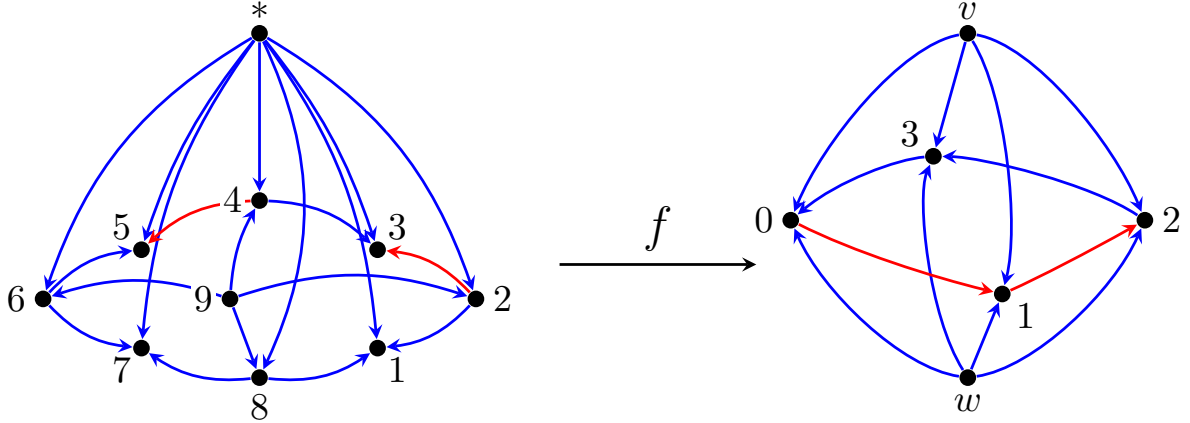
\begin{figure}[H]
\centering
\resizebox{0.98\textwidth}{!}{%
\begin{tikzpicture}[>=stealth]
\tikzset{
    dot/.style={circle,fill=black,inner sep=1.7pt},
    vlabel/.style={fill=white,inner sep=0.8pt,font=\normalsize},
    posedge/.style={->,blue,thick},
    negedge/.style={->,red,thick}
}

\begin{scope}[scale=1.0]
    \node[dot] (s)  at ( 0.0,4.7) {};
    \node[dot] (x6) at (-2.2,2.0) {};
    \node[dot] (x2) at ( 2.2,2.0) {};
    \node[dot] (x4) at ( 0.0,3.0) {};
    \node[dot] (x8) at ( 0.0,1.2) {};
    \node[dot] (a1) at (-1.2,2.5) {};
    \node[dot] (a2) at ( 1.2,2.5) {};
    \node[dot] (b1) at (-1.2,1.5) {};
    \node[dot] (b2) at ( 1.2,1.5) {};
    \node[dot] (x9) at (-0.3,2.0) {};

    \draw[posedge,bend right=20] (s) to (x6);
    \draw[posedge,bend left=20]  (s) to (x2);
    \draw[posedge]              (s) -- (x4);
    \draw[posedge,bend left=25]  (s) to (x8);
    \draw[posedge,bend right=10] (s) to (a1);
    \draw[posedge,bend left=10]  (s) to (a2);
    \draw[posedge,bend right=15] (s) to (b1);
    \draw[posedge,bend left=15]  (s) to (b2);

    \draw[posedge,bend left=18]  (x9) to (x2);
    \draw[posedge,bend right=18] (x9) to (x6);
    \draw[posedge]              (x9) -- (x8);
    \draw[posedge,bend left=12]  (x9) to (x4);

    \draw[posedge,bend left=20]  (x6) to (a1);
    \draw[negedge,bend right=20] (x4) to (a1);
    \draw[posedge,bend left=20]  (x4) to (a2);
    \draw[negedge,bend right=20] (x2) to (a2);
    \draw[posedge,bend right=20] (x6) to (b1);
    \draw[posedge,bend left=20]  (x8) to (b1);
    \draw[posedge,bend right=20] (x8) to (b2);
    \draw[posedge,bend left=20]  (x2) to (b2);

    \node[vlabel,anchor=south,yshift=3pt] at (s)  {\(\ast\)};
    \node[vlabel,anchor=east,xshift=-4pt] at (x6) {\(6\)};
    \node[vlabel,anchor=west,xshift=4pt]  at (x2) {\(2\)};
    \node[vlabel,anchor=east,xshift=-4pt] at (x4) {\(4\)};
    \node[vlabel,anchor=south east,xshift=-2pt,yshift=2pt] at (a1) {\(5\)};
    \node[vlabel,anchor=south west,xshift=2pt,yshift=2pt]  at (a2) {\(3\)};
    \node[vlabel,anchor=north east,xshift=-2pt,yshift=-2pt] at (b1) {\(7\)};
    \node[vlabel,anchor=north west,xshift=2pt,yshift=-2pt]  at (b2) {\(1\)};
    \node[vlabel,anchor=east,xshift=-4pt] at (x9) {\(9\)};
    \node[vlabel,anchor=north,yshift=-4pt] at (x8) {\(8\)};
\end{scope}

\draw[->,thick] (3.05,2.35) -- node[above]{\Large \(f\)} (5.05,2.35);

\begin{scope}[xshift=7.20cm,scale=1.0]
    \node[dot] (vtx) at ( 0.00,4.70) {};
    \node[dot] (y3)  at (-0.35,3.45) {};
    \node[dot] (y0)  at (-1.80,2.80) {};
    \node[dot] (y2)  at ( 1.80,2.80) {};
    \node[dot] (y1)  at ( 0.35,2.05) {};
    \node[dot] (wtx) at ( 0.00,1.20) {};

    % Use symmetric Bézier paths so the suspension edges read as a smooth
    % outer frame and two clearly separated central strands.
    \draw[posedge] (vtx) .. controls (-0.68,4.58) and (-1.48,3.58) .. (y0);
    \draw[posedge] (vtx) .. controls ( 0.68,4.58) and ( 1.48,3.58) .. (y2);
    \draw[posedge] (vtx) -- (y3);
    \draw[posedge] (vtx) .. controls ( 0.40,4.05) and ( 0.52,2.72) .. (y1);
    \draw[posedge] (wtx) .. controls ( 0.68,1.30) and ( 1.48,2.08) .. (y2);
    \draw[posedge] (wtx) .. controls (-0.68,1.30) and (-1.48,2.08) .. (y0);
    \draw[posedge] (wtx) -- (y1);
    \draw[posedge] (wtx) .. controls (-0.40,1.83) and (-0.52,3.03) .. (y3);

    \draw[posedge] (y3) .. controls (-0.72,3.42) and (-1.32,3.20) .. (y0);
    \draw[posedge] (y2) .. controls ( 1.25,3.17) and ( 0.15,3.44) .. (y3);
    \draw[negedge] (y0) .. controls (-1.17,2.48) and (-0.35,2.22) .. (y1);
    \draw[negedge] (y1) .. controls ( 0.73,2.22) and ( 1.23,2.48) .. (y2);

    \node[vlabel,anchor=south,yshift=3pt] at (vtx) {\(v\)};
    \node[vlabel,anchor=north,yshift=-4pt] at (wtx) {\(w\)};
    \node[vlabel,anchor=east,xshift=-4pt] at (y0) {\(0\)};
    \node[vlabel,anchor=west,xshift=4pt]  at (y2) {\(2\)};
    \node[vlabel,anchor=south east,xshift=-3pt,yshift=2pt] at (y3) {\(3\)};
    \node[vlabel,anchor=north west,xshift=3pt,yshift=-2pt] at (y1) {\(1\)};
\end{scope}
\end{tikzpicture}%
}
\caption{The signed weak morphism in \cref{ex:suspension-functor} for the signing with negative arrows \(0\to1\) and \(1\to2\) in the suspension.  Blue solid arrows are positive and red solid arrows are negative; the induced signing on the left has negative arrows \(2\to3\) and \(4\to5\).}
\label{fig:suspension-functor}
\end{figure}

There are no allowed three-paths in \(D_L\).  For an allowed three-path
$i_0i_1i_2i_3$ in $D_R$, the vertices $i_1,i_2,i_3$ lie consecutively
on the directed four-cycle.  Deleting $i_2$ produces the non-allowed
face $i_0i_1i_3$.  For the fixed ordered pair $(i_1,i_3)$, there is a unique rim vertex $i_2$
satisfying $i_1\to i_2\to i_3$.  Hence the lifted face obtained by deleting
$i_2$ occurs in the boundary of only one allowed lifted three-path.  The coefficient of that
three-path must therefore be zero in every boundary-preserving chain.
Hence $\Omega_3^{\sigma_L}(D_L)=\Omega_3^{\sigma_R}(D_R)=0$, and
second homology is the kernel of the degree-two boundary in each case.
Writing \(\boldsymbol\beta^\rho(X)=(\beta_0^\rho(X),
\beta_1^\rho(X),\beta_2^\rho(X))\), the resulting Betti vectors are:
\begin{center}
\renewcommand{\arraystretch}{1.15}
\begin{tabular}{ccll}
\toprule
negative arrows in $D_R$ &
negative arrows in $D_L$ &
$\boldsymbol\beta^{\sigma_L}(D_L)$ &
$\boldsymbol\beta^{\sigma_R}(D_R)$\\
\midrule
$\varnothing$ & $\varnothing$ & $(1,0,1)$ & $(1,0,1)$\\
$\{0\to1\}$ & $\{2\to3\}$ & $(0,0,0)$ & $(0,0,0)$\\
$\{0\to1,\;1\to2\}$ & $\{2\to3,\;4\to5\}$ & $(0,2,0)$ & $(0,2,0)$\\
\bottomrule
\end{tabular}
\end{center}
The corresponding chain dimensions and boundary ranks are
\[
\begin{array}{c|ccc}
 &\dim\Omega_2^\sigma
 &\rank\bigl(\partial_1^\sigma\bigr)
 &\rank\bigl(\partial_2^\sigma\bigr)\\\hline
 D_L&(12,10,8)&(9,10,10)&(11,10,8)\\
 D_R&(8,6,4)&(5,6,6)&(7,6,4).
\end{array}
\]
Each tuple follows the order of the three signings in the preceding table;
the maps \(\partial_1^\sigma\) and \(\partial_2^\sigma\) are taken in the
corresponding signed GLMY complex of each digraph.
Together with the zero degree-three spaces, these ranks give the displayed
Betti vectors.
The third signing is drawn in \cref{fig:suspension-functor}.
For each pulled-back signing, the induced maps on homology are supplied
by \cref{thm:signed-digraph-functoriality}.
\end{example}

\section{A higher-dimensional sign-sensitive example}\label{sec:h2-example}

For a finite signed digraph $(D,\sigma)$, construct the ordinary digraph
$\widetilde D_\sigma$ and compute its GLMY chain complex
$\Omega_\bullet(\widetilde D_\sigma)$.  The deck involution $\tau$ acts on every chain space
\[
        \tau:\Omega_p(\widetilde D_\sigma)
        \longrightarrow \Omega_p(\widetilde D_\sigma),
        \qquad \tau^2=\operatorname{id},
\]
and commutes with the boundary.  Over $\mathbb R$, the signed chain group is
identified with the anti-invariant subspace
\[
        \Omega_p^\sigma(D)
        \cong \Omega_p(\widetilde D_\sigma)^-
        =\ker\bigl(\tau+\operatorname{id}\bigr),
\]
and the signed boundary is the restriction of the ordinary GLMY boundary to
this anti-invariant subcomplex.  Equivalently, one may use the anti-invariant generators
$e_{\widetilde\gamma}-e_{\tau\widetilde\gamma}$.  Thus the signed Betti
numbers are obtained from the ordinary boundary matrices of the double cover
by restricting them to the anti-invariant summands.
Since the invariant and anti-invariant chain spaces form a direct-sum
decomposition by subcomplexes, the same reduction holds on homology:
\[
        H_p^\sigma(D)
        \cong H_p(\widetilde D_\sigma)^-
        :=\ker\bigl(\tau_*+\operatorname{id}\bigr).
\]

The ordinary GLMY computation on the double cover can be carried out using
our existing program
\texttt{Digraph.py}\footnote{Chinese Patent Application
No.~CN202610465550.3, \emph{Method and apparatus for high-order graph
structure analysis based on directed path structures, electronic device, and
storage medium}, filed April~9, 2026.  Applicant: Beijing Institute of
Mathematical Sciences and Applications (BIMSA).  Inventors: Xiang Liu,
Jingyan Li, and Jie Wu.}; the signed GLMY homology is then obtained by
restricting the resulting boundary matrices to the deck anti-invariant
subspaces and computing their homology.  If $(D,\sigma)$ is
switching-balanced, \cref{prop:double-cover-splitting,cor:switching-balanced-reduction}
reduce the computation further to ordinary GLMY homology on $D$ itself.

\begin{example}\label{ex:h2-sign-sensitive}
Let $D$ have vertices $\{0,1,2,3,4\}$ and arrows
\[
        0\to1,\quad 1\to2,\quad 2\to0,
        \qquad
        0,1,2\to3,
        \qquad
        0,1,2\to4.
\]
Thus $0\to1\to2\to0$ is a directed $3$-cycle, and the vertices $3$ and $4$ are two common sinks.  In the ordinary GLMY complex this digraph has
\[
        (\beta_0,\beta_1,\beta_2)=(1,0,1).
\]
One representative of the ordinary $2$-class is
\[
        -e_{013}+e_{014}
        -e_{123}+e_{124}
        -e_{203}+e_{204}.
\]
The displayed two-cycle is the difference between the cone chains at
the sinks $4$ and $3$ over the directed cycle $0\to1\to2\to0$.

The arrows $01,12,03,04$ form a spanning tree of the underlying graph.
Every switching class contains a unique normalized signing satisfying
\[
        \sigma(01)=\sigma(12)=\sigma(03)=\sigma(04)=1.
\]
Write
\[
        x=\sigma(20),\quad
        y=\sigma(13),\quad
        z=\sigma(23),\quad
        r=\sigma(14),\quad
        s=\sigma(24).
\]
The signed two-path generators that can contribute to
$\Omega_2^\sigma(D)$ lie over the six cone triangles
\[
        013,\quad 014,\quad 123,\quad 124,\quad 203,\quad 204.
\]
In this computation, a base-path label \(\gamma\) denotes the signed
quotient generator \([e_{\widetilde\gamma^+}]\), where
\(\widetilde\gamma^+\) is its lift starting in the positive sheet.
In particular, the same convention is used for the row and column bases of
the boundary matrix below.  In the displayed order, these six classes belong to
\(\Omega_2^\sigma(D)\) under the respective conditions
\[
        y=1,\qquad r=1,\qquad z=y,\qquad s=r,\qquad z=x,\qquad s=x.
        \tag{*}
\]
Each cycle path $012,120,201$ has a non-allowed shortcut face that
occurs in the boundary of no other allowed two-path.  For any cone path whose condition in $(*)$ fails, its non-allowed lifted
shortcut likewise occurs in the boundary of no other allowed two-path.  Therefore every chain in
$\Omega_2^\sigma(D)$ has zero coefficients on the three cycle paths
and on all incompatible cone paths.  The compatible cone generators
form a basis of $\Omega_2^\sigma(D)$.  If $t$ is the number of
satisfied equations in $(*)$, then
\[
        \dim \Omega_2^\sigma(D)=t.
\]
The six equations in $(*)$ are equivalent, in the displayed order, to
requiring the six signs
\[
        y,\quad r,\quad yz,\quad rs,\quad zx,\quad sx
\]
to equal $+1$.  Their product is always $1$.  Conversely, given six signs
\((a_1,\ldots,a_6)\) with product \(1\), the assignments
\[
 y=a_1,\quad r=a_2,\quad z=a_1a_3,\quad
 s=a_2a_4,\quad x=a_1a_3a_5
\]
determine the unique normalized tuple $(x,y,z,r,s)$; the equation
\(sx=a_6\) follows from the product constraint.  Thus the number of failed equations
in \((*)\) is even.  Consequently
\[
        t=6,4,2,0
\]
and the numbers of switching classes for $t=6,4,2,0$, respectively, are
\[
        \binom60,\quad \binom62,\quad \binom64,\quad \binom66.
\]
The underlying graph is connected, so the kernel of the vertex-switching
action consists of the two constant functions.  Each switching class
therefore has \(2^{|V|-1}=16\) signings.  Thus the cases $t=6,4,2,0$
contain $16,240,240,16$ signings and $1,15,15,1$ switching classes,
respectively.  The rank calculation below shows that the four values
of $t$ correspond to four distinct signed Betti vectors.

\begin{samepage}
Order the row basis by the signed degree-one generators associated with
the base arrows
\[
01,12,20,03,13,23,04,14,24,
\]
and order the six potential cone columns by the base two-paths
\[
013,014,123,124,203,204.
\]
Here each base-path label \(\gamma\) stands for the signed quotient
generator \([e_{\widetilde\gamma^+}]\) introduced before the compatibility
conditions $(*)$.  With these
ordered bases, the degree-two boundary matrix is obtained by retaining
the compatible columns of
\[
M_x=\begin{pmatrix}
1&1&0&0&0&0\\
0&0&1&1&0&0\\
0&0&0&0&1&1\\
-1&0&0&0&x&0\\
1&0&-1&0&0&0\\
0&0&1&0&-1&0\\
0&-1&0&0&0&x\\
0&1&0&-1&0&0\\
0&0&0&1&0&-1
\end{pmatrix}.
\]
\end{samepage}
A vector $(c_1,\ldots,c_6)\in\mathbb R^6$ belongs to $\ker M_x$ precisely when,
for some $a\in\mathbb R$,
\[
        (c_1,c_2,c_3,c_4,c_5,c_6)
        =a(1,-1,1,-1,1,-1),\qquad (x-1)a=0.
\]
Thus \(M_1\) has rank \(5\), with a unique relation involving all six
columns, and \(M_{-1}\) has rank \(6\).  Every proper subset of the
columns is independent.  All six compatibility conditions hold precisely
when \(x=y=z=r=s=1\).

In fact, \(\Omega_p^\sigma(D)=0\) for every \(p\ge3\).  Let
$i_0\cdots i_p$ be an allowed $p$-path with $p\ge3$.  Since $3$ and $4$
are sinks, the vertices $i_0,i_1,i_2$ lie consecutively on the directed
three-cycle.  Deleting $i_1$ produces the regular face $i_0i_2\cdots i_p$,
whose first step $i_0i_2$ is not an arrow.  The intermediate vertex
$i_1$ is uniquely determined by $i_0$ and $i_2$, so the corresponding
lifted face occurs in the boundary of only one allowed lifted $p$-path.
Consequently, every allowed lifted $p$-path has coefficient zero in a
boundary-preserving signed chain, and $\Omega_p^\sigma(D)=0$.

The twisted incidence matrix has rank \(4\) when the signing is
switching-balanced and rank \(5\) otherwise: a vector in its left kernel
must satisfy \(h(j)=\sigma(ij)h(i)\) on each arrow, and these equations
have a nonzero solution precisely in the switching-balanced case.
In the normalized coordinates, the signing is switching-balanced
exactly when \(t=6\).  Moreover, every signed arrow is boundary-preserving,
so $\dim\Omega_0^\sigma(D)=5$ and $\dim\Omega_1^\sigma(D)=9$ for every
signing.  Using the degree-specific signed boundaries
\(\partial_1^\sigma\) and \(\partial_2^\sigma\) defined in
\cref{def:signed-glmy-complex}, we obtain
\[
\begin{aligned}
\beta_0^\sigma
&=5-\rank\bigl(\partial_1^\sigma:
   \Omega_1^\sigma(D)\to\Omega_0^\sigma(D)\bigr),\\
\beta_1^\sigma
&=9-\rank\bigl(\partial_1^\sigma:
   \Omega_1^\sigma(D)\to\Omega_0^\sigma(D)\bigr)
   -\rank\bigl(\partial_2^\sigma:
   \Omega_2^\sigma(D)\to\Omega_1^\sigma(D)\bigr),\\
\beta_2^\sigma
&=\dim\Omega_2^\sigma(D)
  -\rank\bigl(\partial_2^\sigma:
   \Omega_2^\sigma(D)\to\Omega_1^\sigma(D)\bigr).
\end{aligned}
\]
Consequently, the \(2^9\) signings yield exactly the following four Betti
vectors, and their signed homology vanishes in every degree \(p\ge3\).
\begin{center}
\begin{tabular}{ccl}
\toprule
$(\beta_0^\sigma,\beta_1^\sigma,\beta_2^\sigma)$
& number of signings & one representative set of negative arrows\\
\midrule
$(1,0,1)$ & 16  & $\varnothing$\\
$(0,0,0)$ & 240 & $\{2\to4\}$\\
$(0,2,0)$ & 240 & $\{2\to3,\;2\to4\}$\\
$(0,4,0)$ & 16  & $\{2\to0,\;1\to3,\;1\to4\}$\\
\bottomrule
\end{tabular}
\end{center}

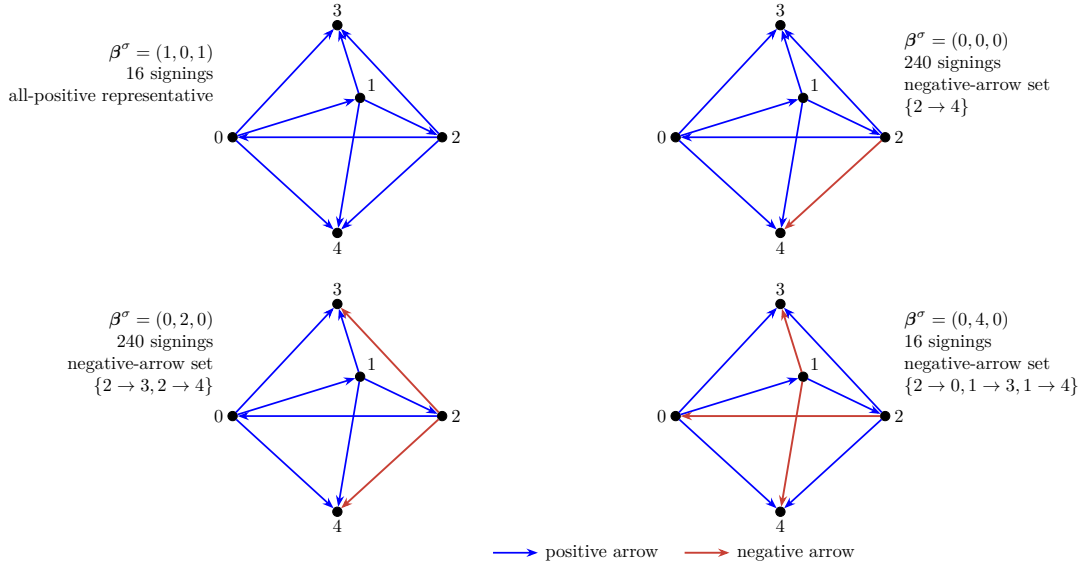
\begin{figure}[htbp]
\centering
\resizebox{0.9\textwidth}{!}{%
\begin{tikzpicture}[x=1.15cm,y=1.05cm,node distance=1mm]
\definecolor{signnegred}{RGB}{200,64,53}
\tikzset{
  htwopos/.style={-{Stealth[length=2.2mm,width=1.6mm]},draw=blue,line width=.95pt},
  htwoneg/.style={-{Stealth[length=2.2mm,width=1.6mm]},draw=signnegred,line width=1.05pt},
  htwovertex/.style={circle,draw=black,fill=black,line width=.35pt,inner sep=1.8pt},
  htwolabel/.style={font=\small,fill=white,inner sep=.6pt}
}
\newcommand{\hTwoVertices}{%
  \node[htwovertex] (v0) at (-1.75,0) {};
  \node[htwovertex] (v1) at (.38,.72) {};
  \node[htwovertex] (v2) at (1.75,0) {};
  \node[htwovertex] (v3) at (0,2.05) {};
  \node[htwovertex] (v4) at (0,-1.75) {};
  \node[htwolabel,left=1.5pt of v0] {$0$};
  \node[htwolabel,above right=1.5pt of v1] {$1$};
  \node[htwolabel,right=1.5pt of v2] {$2$};
  \node[htwolabel,above=1.5pt of v3] {$3$};
  \node[htwolabel,below=1.5pt of v4] {$4$};
}
\newcommand{\hTwoEdge}[3]{\draw[#1] (#2)--(#3);}

\node[font=\large] at (0,6.05)
  {Same digraph, different signs, different signed GLMY homology groups};

\begin{scope}[shift={(-3.7,2.45)}]
  \node[align=right,font=\small,anchor=east] at (-1.95,1.15)
    {$\boldsymbol\beta^\sigma=(1,0,1)$\\16 signings\\all-positive representative};
  \hTwoVertices
  \hTwoEdge{htwopos}{v0}{v1}
  \hTwoEdge{htwopos}{v1}{v2}
  \hTwoEdge{htwopos}{v2}{v0}
  \hTwoEdge{htwopos}{v0}{v3}
  \hTwoEdge{htwopos}{v1}{v3}
  \hTwoEdge{htwopos}{v2}{v3}
  \hTwoEdge{htwopos}{v0}{v4}
  \hTwoEdge{htwopos}{v1}{v4}
  \hTwoEdge{htwopos}{v2}{v4}
\end{scope}

\begin{scope}[shift={(3.7,2.45)}]
  \node[align=left,font=\small,anchor=west] at (1.95,1.15)
    {$\boldsymbol\beta^\sigma=(0,0,0)$\\240 signings\\negative-arrow set\\$\{2\to4\}$};
  \hTwoVertices
  \hTwoEdge{htwopos}{v0}{v1}
  \hTwoEdge{htwopos}{v1}{v2}
  \hTwoEdge{htwopos}{v2}{v0}
  \hTwoEdge{htwopos}{v0}{v3}
  \hTwoEdge{htwopos}{v1}{v3}
  \hTwoEdge{htwopos}{v2}{v3}
  \hTwoEdge{htwopos}{v0}{v4}
  \hTwoEdge{htwopos}{v1}{v4}
  \hTwoEdge{htwoneg}{v2}{v4}
\end{scope}

\begin{scope}[shift={(-3.7,-2.65)}]
  \node[align=right,font=\small,anchor=east] at (-1.95,1.15)
    {$\boldsymbol\beta^\sigma=(0,2,0)$\\240 signings\\negative-arrow set\\$\{2\to3,2\to4\}$};
  \hTwoVertices
  \hTwoEdge{htwopos}{v0}{v1}
  \hTwoEdge{htwopos}{v1}{v2}
  \hTwoEdge{htwopos}{v2}{v0}
  \hTwoEdge{htwopos}{v0}{v3}
  \hTwoEdge{htwopos}{v1}{v3}
  \hTwoEdge{htwoneg}{v2}{v3}
  \hTwoEdge{htwopos}{v0}{v4}
  \hTwoEdge{htwopos}{v1}{v4}
  \hTwoEdge{htwoneg}{v2}{v4}
\end{scope}

\begin{scope}[shift={(3.7,-2.65)}]
  \node[align=left,font=\small,anchor=west] at (1.95,1.15)
    {$\boldsymbol\beta^\sigma=(0,4,0)$\\16 signings\\negative-arrow set\\$\{2\to0,1\to3,1\to4\}$};
  \hTwoVertices
  \hTwoEdge{htwopos}{v0}{v1}
  \hTwoEdge{htwopos}{v1}{v2}
  \hTwoEdge{htwoneg}{v2}{v0}
  \hTwoEdge{htwopos}{v0}{v3}
  \hTwoEdge{htwoneg}{v1}{v3}
  \hTwoEdge{htwopos}{v2}{v3}
  \hTwoEdge{htwopos}{v0}{v4}
  \hTwoEdge{htwoneg}{v1}{v4}
  \hTwoEdge{htwopos}{v2}{v4}
\end{scope}

\draw[htwopos] (-1.1,-5.15)--(-.35,-5.15)
  node[right,black,font=\small]{positive arrow};
\draw[htwoneg] (2.1,-5.15)--(2.85,-5.15)
  node[right,black,font=\small]{negative arrow};
\end{tikzpicture}%
}
\caption{The fixed digraph in \cref{ex:h2-sign-sensitive} with four
representative signings.  Arrowheads specify the common orientation; blue
solid edges are positive and red solid edges are negative.  Varying only the edge signs
changes the signed GLMY Betti vector from \((1,0,1)\) to \((0,0,0)\),
\((0,2,0)\), or \((0,4,0)\).}
\label{fig:h2-signing-classes}
\end{figure}

For the fixed digraph $D$ in this example, the second signed homology
is one-dimensional precisely
for switching-balanced signings and vanishes for every other signing.
Among the unbalanced signings, the first signed homology has dimension
$0$, $2$ or $4$, while the zero-dimensional group vanishes in each of the
three unbalanced Betti-vector cases listed above.
\end{example}

\FloatBarrier
\section{Conclusion and outlook}\label{sec:conclusion}

The signed GLMY complex is naturally isomorphic to the deck
anti-invariant subcomplex of the ordinary GLMY chain complex on the signed
double cover.  The double-cover realization yields switching invariance and
recovers ordinary GLMY homology for switching-balanced signings.  Bidirected
completion defines canonical signed graph homology, with
\[
 \mathcal H_p^\sigma(\Sigma)=H_p^\sigma(\overleftrightarrow\Sigma),
 \qquad \mathcal H_0^\sigma(\Sigma)\cong\ker L^\sigma(\Sigma).
\]
The zero-dimensional identification uses the standard vertex inner
product.  For signed digraphs, the all-positive reduction retains the
orientation sensitivity of ordinary GLMY homology.  The computations in
\cref{ex:unbalanced-triangle,ex:canonical-sign-sensitive-bowtie,ex:h2-sign-sensitive}
show, in addition, that changing only the signing can alter the first and
second signed GLMY homology groups.  Signed weak morphisms induce maps on
signed GLMY homology through deck-equivariant maps of covers.

Further questions include combinatorial descriptions and spectral estimates
for the signed GLMY Laplacians, as well as stability under weighted
filtrations.  Extensions to gain graphs and higher-rank local systems, a
compatible homotopy theory, and a realization theorem remain open.  For
signed regulatory networks, one may ask which homology classes persist under
perturbations of arrow weights and signs.

\section*{Acknowledgments}

Shuliang Bai was supported by the National Natural Science Foundation of China
(NSFC, Grant No.~12301434).  Jingyan Li was supported by the Beijing Natural
Science Foundation (International Scientists Project, Grant No.~IS25081) and
the National Natural Science Foundation of China (NSFC, General Program, Grant
No.~82573048).

% Use the supplied bibliography directly; no .bib file or BibTeX run is required.

\end{document}